%% file: main.tex
\documentclass[a4paper, 12pt,pdftex
]{amsart}

\usepackage[all]{xy}
\usepackage[latin1]{inputenc}
\usepackage{amsmath,amssymb,amsfonts, amsthm, amscd, latexsym, marginnote}

\usepackage{rotating}
\usepackage[english]{babel}
\usepackage[mathcal]{euscript}

\usepackage[hyperindex,breaklinks,pdftex]{hyperref} 

\usepackage{graphicx}
\CompileMatrices

\newtheorem{thm}{Theorem}[section]

\newtheorem{cor}[thm]{Corollary}
\newtheorem{lem}[thm]{Lemma}
\newtheorem{prop}[thm]{Proposition}
\theoremstyle{definition}

\theoremstyle{definition}

\newtheorem{rmk}[thm]{Remark}

\newcommand\Q{\mathbb Q}

\newcommand\g{\gamma}

\newcommand\G{\Gamma}

\newcommand\Id{\mathrm{Id}}

\newcommand{\R}{{\mathbb R}}
\newcommand{\C}{{\mathbb C}}
\newcommand{\CP}{{\mathbb{CP}}}
\newcommand{\Z}{{\mathbb Z}}

\newcommand{\F}[0]{\mathbb{F}}

\newcommand{\N}[0]{\mathbb{N}}

\newcommand{\into}[0]{\hookrightarrow}

\newcommand{\im}[0]{\mathrm{Im }}

\newcommand{\SL}[0]{SL_{2}(\Z)}
\newcommand{\Ar}[0]{\Z_{p^r}}

\def\qed{\ifmmode $\Box$ \else{\unskip\nobreak\hfil
\penalty50\hskip1em\null\nobreak\hfil $\Box$
\parfillskip=0pt\finalhyphendemerits=0\endgraf}\fi}

\newcommand{\eq}[1][r]
       {\ar@<-3pt>@{->}[#1]
        \ar@<-1pt>@{}[#1]|<{}="gauche"
        \ar@<+0pt>@{}[#1]|-{}="milieu"
        \ar@<+1pt>@{}[#1]|>{}="droite"
        \ar@/^2pt/@{-}"gauche";"milieu"
        \ar@/_2pt/@{-}"milieu";"droite"}
\newcommand{\imm}[1][r] {\ar@{^{(}->}[#1]}

\newcommand{\Ga}[0]{\Gamma}
\newcommand{\first}[0]{\mathcal{P}}
\newcommand{\second}[0]{\mathcal{Q}}
\newcommand{\coker}[0]{\mathrm{coker}}

\def\g#1{\save
[].[drr]="g#1"*[F-,]\frm{}\restore}%

\def\G#1{\save
[].[drrr]="g#1"*[F-,]\frm{}\restore}%

\def\H#1{\save
[].[drrr]="g#1"*[F-,]\frm{}\restore}%

\begin{document}

\title[The cohomology of the braid group with coefficients ]{The cohomology of the braid group $B_3$ and of $SL_2(\mathbb Z)$ with coefficients in a geometric representation}

\author[F.~Callegaro]{F.~Callegaro}
\address{Scuola Normale Superiore, Pisa, Italy}
\email{f.callegaro@sns.it}

\author[F.~R.~Cohen]{F.~R.~Cohen$^{*}$}
\address{Department of Mathematics,
University of Rochester, Rochester NY 14627, USA}
\email{cohf@math.rochester.edu}
\thanks{$^{~*}$Partially supported by DARPA}

\author[M.~Salvetti]{M.~Salvetti}
\address{Department of Mathematics,
University of Pisa, Pisa Italy}
\email{salvetti@dm.unipi.it}

\date{\today}
\input{abstract.tex}

\maketitle

\input{intro.tex}

\input{general.tex}

\input{results.tex}

\input{poly_inv.tex}

\input{cohom246.tex}

\input{sl2modp.tex}

\input{cohomB3.tex}

\input{comparisons.tex}
\input{appendix.tex}

\providecommand{\bysame}{\leavevmode\hbox
to3em{\hrulefill}\thinspace}
\providecommand{\MR}{\relax\ifhmode\unskip\space\fi MR }
\providecommand{\MRhref}[2]{%
  \href{http://www.ams.org/mathscinet-getitem?mr=#1}{#2}
} \providecommand{\href}[2]{#2}

\bibliographystyle{amsalpha}
\bibliography{biblio}
\nocite{*}


\end{document}

%% file: abstract.tex
\begin{abstract}
The purpose of this article is to describe the integral cohomology of the braid group $B_3$ and $\SL$ with local coefficients in a classical geometric representation given by symmetric powers of the natural symplectic representation. 

These groups have a description in terms of the so called ``divided polynomial algebra''. The results show a strong relation between torsion part of the computed cohomology and fibrations related to loop spaces of spheres.
 
\end{abstract}

%% file: intro.tex
\section{Introduction} \label{introduction}

This article  addresses the cohomology of the third braid group $B_3$ or $\SL$  with local coefficients in a classical geometric representation given by symmetric powers of the natural symplectic representation.  Calculations of this type have been basic in several areas of mathematics such as number theory as well as algebraic geometry. These results can be regarded as giving the cohomology of certain mapping class groups with non-trivial local coefficients.

For example, the first characteristic zero computation of this type, important in number theory, is the classical computation of a certain ring of modular forms due to G.~Shimura \cite{Shimura}. A version for the ``stabilized'' mapping class group in characteristic zero was developed by Looijenga \cite{looij96}. The new feature of the current article is (1) a determination of the torsion in the cohomology of this special case of $B_3$ together with (2) an identification of this torsion as the cohomology of another space which has been basic within homotopy theory.

The main new results in this paper gives the cohomology of the braid group $B_3$ as well as $SL(2,\Z)$ with natural coefficients over integral symmetric powers of the natural symplectic representation. One of the main new points is that the construction of Shimura which gives the ring of modular forms also ``transforms'' polynomial rings to torsion variations of divided power algebras closely connected to the cohomology of the loop space of a sphere and related fibrations.

The precise problem considered here is the computation of the cohomology of
the braid group $B_3$ and of the special linear group $\SL$ with
certain non-trivial local coefficients.  First consider a particularly natural geometric, classical representation. 
In general, let $M_{g,n}$ be an orientable surface of genus $g$ with $n$ connected components in its
boundary.  Isotopy classes of  Dehn twists around simple loops $c_1,\dots, c_{2g}$
such that $$|c_i\cap c_{i+1}|=1 \ (i=1,\dots,2g-1),\  c_i\cap
c_j=\emptyset\ \text{otherwise}$$ 
represents the braid group $B_{2g+1}$ 
in the symplectic group
$Aut(H_1(M_{g,n});$ $ <>)$  of all automorphisms preserving the
intersection form. Such representations arise classically, for
example in singularity theory, as it is related to some monodromy
operators (see \cite{Wajnryb} for example) as well as in number theory in the guise
of modular forms (see \cite{Shimura} for example), and geometry (see \cite{Looijenga} for example). 

In this paper we restrict to the case $g=1,\ n=1,$ where the symplectic
group equals $\SL.$  We extend the above
representation to the symmetric power $M=\Z[x,y].$ Notice that 
this representation splits into irreducibles $M=\oplus_{n\geq 0}\ M_n,$ according to
polynomial-degree.  

Our result is the complete integral computation of  $$H^*(B_3;M)\ =\ \oplus_n\ H^*(B_3;M_n)$$
and $$H^*(\SL;M)\ =\ \oplus_n\ H^*(\SL;M_n).$$ The free part of the cohomology of $\SL$ is
a classical computation due to G.~Shimura 
(\cite{Shimura}, \cite{fty}) and is related to certain spaces of modular
forms. Here we give the complete calculation of the cohomology including the
$p$-torsion groups. 

These groups have an interesting description in terms of 
a variation of the so called ``divided polynomial algebra''. 
Unexpectedly, we found a strong  relation between the cohomology 
in this case and the cohomology of some spaces which
were constructed in a completely different framework and which are basic
to the growth of $p$-torsion in the homotopy groups of spheres.   

An outline of the paper is as follows. In section \ref{general}, we give the general results and the main tools which 
are used here.  Section \ref{results} is devoted to giving all precise statements together with some additional
details concerning a mod $p$ variation of the divided polynomial algebra.  In sections \ref{invariants},  \ref{cohom246}, \ref{s:sl2modp},  \ref{s:cohomB3} we develop all the technical tools for the computations. In particular, section \ref{invariants} is devoted to the first cohomology of $\SL$ and of $B_3,$ where we use some classical results on polynomial invariants under $\SL;$ in section \ref{cohom246} we compute polynomial invariants under the $\mathbb Z_2,\ \mathbb Z_4,\ \mathbb Z_6$- subgroups of $\SL;$  section \ref{s:sl2modp} is devoted to the $p-$torsion of the higher cohomology of $\SL;$ in section \ref{s:cohomB3} we use the previous results and the associated spectral sequence for the full computation of the second cohomology of $B_3.$ Section \ref{Topological comparisons and speculations} compares the cohomology determined here to the cohomology of some spaces which seem quite different. We are led to ask questions about possible connections as well as questions about analogous
results concerning congruence subgroups as well as extending these results to $B_n$ for all $n$.

These structures are at the intersection of several basic topics during the special period on ``Configuration Spaces'', which took place in Pisa. The authors started to work together to develop the problem considered here precisely during the special period on ``Configuration Spaces'' in Pisa.  

The authors are grateful to the  Ennio de Giorgi Mathematical Research Institute, the Scuola Normale Superiore as well as the University of Pisa for the opportunity to start this work together and the wonderful setting to do some mathematics; they would like to thank Pierre Vogel for suggesting
the proof of proposition \ref{p:ssmayerviet}.
The second author would like to thank many friends for an interesting, wonderful time in Pisa.

%% file: general.tex
\section{General results and main tools}\label{general}
Let  $B_3$ denote the braid group with $3$ strands. We consider the geometric representation of $B_3$ into the group 
 $$Aut(H_1(M_{1,1};\Z),<\ >)\ \cong\ SL_2(\Z)$$
 of the automorphisms which preserve the intersection form where
$M_{1,1}$ is a genus-$1$ oriented surface with $1$ boundary component.  
 Explicitly, if the standard presentation of $B_3$ is
$$B_3\ =\ <\sigma_1,\ \sigma_2\ : \  \sigma_1\sigma_2\sigma_1\ =\ \sigma_2\sigma_1\sigma_2>,$$
then one sends $\sigma_1,\ \sigma_2$ into the Dehn twists around one parallel and one meridian of $M_{1,1}$ respectively. 
Taking a natural basis  for the $H_1,$ the previous map is explicitly given by 
$$\lambda : B_3\ \to \     \SL :
\quad \sigma_1\to \begin{bmatrix} 1&0\\-1&1\end{bmatrix},\quad  \sigma_2\to \begin{bmatrix} 1& 1\\ 0& 1\end{bmatrix}. $$ 

Of course, any representation of $\SL$ will induce a representation of $B_3$ by composition with $\lambda.$ We will 
ambiguously identify $\sigma_i$ with $\lambda(\sigma_i)$ when the meaning is clear from the context.

It is well known that $\lambda$ is surjective and the kernel is infinite cyclic generated by the element 
$c:=(\sigma_1\sigma_2)^6 $  (\cite[Th. 10.5]{milnor71})
so that we have a sequence of groups

\begin{equation} \label{fibration}
\begin{CD}
1  @>>> \Z  @>j>>  B_3 @>\lambda>> \SL @>>> 1.
\end{CD}
\end{equation}
Let $V$ be a rank-two free $\Z-$module and $x,y$ a basis of $V^*:=Hom_{\Z}(V,\Z).$  The natural action of $\SL$  over $V^*$ 
 $$\sigma_1:\left\{ \begin{array}{l} x\to x-y\\ y \to y\end{array}\right. ,\  \sigma_2:\left\{ \begin{array}{l} x\to x\\ y \to x+ y\end{array} \right. $$
extends to the symmetric algebra $S[V^*],$ which we identify with the polynomial algebra $M:=\Z[x,y].$ If we fix $\deg x  = \deg y = 1$, the homogeneous component $M_n$ of degree $n$  is 
an irreducible invariant free submodule of rank $n+1$ 
so we have a decomposition  $M= \bigoplus_{n}M_n.$ 

 In this paper we describe the cohomology groups
 $$H^*(B_3;M)\ =\ \bigoplus_{n \geq 0}H^*(B_3;M_n)$$
 and
$$H^*(\SL;M)\ =\ \bigoplus_{n \geq 0}H^*(\SL;M_n)$$
where the action of $B_3$ over $M$ is induced by the above map $\lambda.$ 

The free part of the cohomology of $\SL$ tensored with the real numbers is well-known to be isomorphic to the ring 
of modular forms based on the $SL_2(\Z)$-action on the upper $1/2$ plane by fractional linear tranformations
(\cite{fty}):
$$\begin{array}{lllcr} 
H^1(\SL;\ M_{2n}\otimes \R) & \cong & {\mathcal Mod}_{2n+2}^0\oplus \R& &n\geq 1\\
H^i(\SL;\ M_{n}\otimes \R) & = & 0 && \text{  $i> 1$ or $i=1,$ odd $n$} 
\end{array}$$
where ${\mathcal Mod}_{n}^0$ is the vector space of cuspidal modular forms of weight $n.$ In our computations, a basic  
first step is to rediscover the Poincar\'e series for the cohomology of $\SL$ with coefficients in $M$.

Essentially, we use two "general" tools for the computations. The first one is the spectral sequence associated to (\ref{fibration}), with

\begin{equation}\label{spectral}
E_2^{s,t}\ =\ H^s(\SL;H^t(\Z,M))\  \Rightarrow\ H^{s+t}(B_3;M)
\end{equation}
Notice that  the element $c$ above acts trivially on $M,$ so we have
$$ H^i(\Z;M)\ =\ H^i(\Z;\Z)\otimes M\ =\ \left\{\begin{array}{lc} M & \text{for $i=0,\ 1$}\\ 0 & 
\text{for $i>1$}\end{array}\right.$$ and the spectral sequence has a two-row $E_2$-page
\begin{equation}\label{tworows}
E_2^{s,t}\ =\ \left\{\mbox{$\begin{array}{ll} 0 & \text{ for $t>1$}\\ & \\  H^s(\SL,M) & \text{ for $t=0,\ 1.$} \end{array}$}\right.\end{equation} 

The second main tool which we use comes from the following well-known presentation of $\SL$ as an amalgamated product of torsion groups. We set for brevity $\Z_n:=\Z/n\Z.$

\begin{prop}(\cite{mks,serre})\label{presentation}
The group $SL_2(\Z)$ is
an amalgamated free product
$$\SL\ =\ \Z_4\ *_{\Z_2}\ \Z_6$$
where 
$$\begin{array}{ll}
\Z_4\ & \text{is generated by the $4$-torsion element $w_4:=\begin{bmatrix}0 & 1\\ -1 & 0\end{bmatrix}$}\\
 & \\
\Z_6 & \text{is generated by the $6$-torsion element $w_6:=\begin{bmatrix} 1 & 1 \\ -1 & 0 \end{bmatrix}$}\\
 & \\
\Z_2 & \text{is generated by the $2$-torsion element $w_2:=\begin{bmatrix} -1 & 0 \\ 0 & -1 \end{bmatrix}$}.\\

\end{array}$$

\vspace{-\baselineskip}\qed
\end{prop} 


So, for every $\SL$-module $N,$ we can use the associated a Mayer-Vietoris sequence

\begin{equation}\label{mayervietoris}
\begin{array}{l}
 \to  H^i(\Z_4;N)\oplus H^i(\Z_6;N) \to  
  H^i(\Z_2;N)  \to  H^{i+1}(\SL;N)  \to  
\end{array}
\end{equation} 
We start to deduce some "general" result on the cohomology of $\SL.$ From (\ref{mayervietoris}) and very well-known properties of cohomology groups,  it  immediately follows:
\begin{prop}\label{biggertorsion}
Assume that $2$ and $3$ are invertible in the module $N$ (equivalent, $1/6\in N$).  Then 
$$H^i(\SL;N)\ =\ 0 \quad \text{ for $i>1.$}$$

\vspace{-\baselineskip}
\qed
\end{prop}

\begin{cor}\label{bigtor}
If \hfill $1/6\in N$\hfill  and \hfill  $N$ \hfill has no\hfill $\SL$\hfill  invariants (i.e. \ $H^0(\SL;N)=0$) then 
$$H^1(B_3;N)\ =\ H^1(\SL;N)\ =\ H^2(B_3;N).$$
\end{cor}

\begin{proof} 
The corollary 
follows immediately 
from the above spectral sequence which degenerates at $E_2.$
\end{proof}

\begin{rmk} \label{dimcoho} The group $B_3$ has cohomological dimension $2$ (see for example \cite{salvetti94, decsal96}) so we only need to determine cohomology up to $H^2.$ It also follows that the differential
$$d_2^{s,1}:\ E_2^{s,1}\ \to\ E_2^{s+2,0}$$ 
is an isomorphism for all $s\geq 2,$ and is surjective for $s=1.$
\end{rmk}    

For a given finitely generated $\Z$-module $L$ and for a prime $p$ 
define the $p$-torsion component of $L$ as follows: 
$$L_{(p)} =\{x \in L \mid \exists k \in \N \mbox{ such that }p^k x = 0\}.$$ Moreover 
write $FL$ for the free
part of $L$, that is $FL$ defines the isomorphism class of a maximal free $\Z$-submodule of $L$.

Return to the
module $M$ defined before. 
The next result follows directly from the previous discussion (including remark (\ref{dimcoho})). 

\begin{cor} 
\phantomsection 
\label{cor:general}
\begin{enumerate}
\item For all primes $p$ we have \label{torsion1}
$$H^1(B_3;M)_{(p)}\ =\ H^1(\SL;M)_{(p)}.$$
\item For all primes $p>3$,
$$H^1(B_3;M)_{(p)}\ =\ H^1(\SL;M)_{(p)}\ =\ H^2(B_3;M)_{(p)}.$$
\item For the free parts 
for $n >0$, 
$$FH^1(B_3;M_n)\ =\ FH^1(\SL;M_n)\ =\ FH^2(B_3;M_n).$$
\end{enumerate}

\vspace{-\baselineskip}
\qed
\end{cor}

Notice that  statement
(\ref{torsion1})
of Corollary \ref{cor:general} fails for the torsion-free summand: see theorem \ref{t:H^1} below for the precise statement including it.

%% file: results.tex
\section{Main results}\label{results}
In this section we state the main results, namely the complete
description of the cohomology of the braid group $B_3$ and of $\SL$
with coefficients in the module $M.$ First, we need some definitions. 

Let $\Q[x]$ be the ring of rational polynomials in one variable. We
define the subring $$\Gamma[x]:= \Gamma_{\Z}[x] \subset \Q[x]$$  
as the subset generated, as a $\Z$-submodule, by the elements
$x_n:=\ \frac{x^n}{n!}$ for $n \in \N$. It follows from the next
formula that $\Gamma[x]$ is a subring with 
\begin{equation}\label{dividedpoly}
 x_i x_j\  =\   \binom{i+j}{i} x_{i+j}.
 \end{equation}
The algebra $\Gamma[x]$ is usually known as the \emph{divided polynomial algebra}
over $\Z$ (see for example chapter 3C of \cite{hatcher02} or the Cartan seminars
as an early reference \cite{cartan}). 
In the same way, by using ( \ref{dividedpoly} ), one can define
$\Gamma_R[x]:=\Gamma[x]\otimes R$ over any ring $R.$ 


Let $p$ be a prime number. Consider the $p-$adic valuation
\ $v:=v_p:\N\setminus\{0\}\to \N$ such that $p^{v(n)}$ is the maximum
power of $p$ dividing $n.$  
 Define the ideal  $I_p$  of $\Gamma[x]$ as
$$I_p:= \ (p^{v(i)+1}\ x_i,\quad i\geq 1)$$
and call the quotient  $$ \Gamma_p[x] := \Gamma[x]/I_p$$
the \emph{ $p$-local divided polynomial algebra} $(\!\!\!\!\mod p)$.

\begin{prop}\label{prop:algebra}
The ring $\Gamma_p[x]$ is naturally isomorphic to the quotient $$\Z[\xi_1,\xi_p, \xi_{p^2}, \xi_{p^3}, \ldots]/J_p$$ where $J_p$ is the ideal generated by the polynomials  $$ p\xi_1\ , \quad
\xi_{p^i}^p - p \xi_{p^{i+1}}\  (i \geq 1 ).$$
The element $\xi_{p^i}$ corresponds to the generator $x_{p^{i}} \in \Gamma_p[x],\ i\geq 0$.
\end{prop}
\begin{proof} 
The ideal $I_p$ is 
graded and 
there is a direct sum decomposition of $\Z$-modules with respect to the degree
$$\Gamma_p[x]\ =\ \oplus_{i\geq 0}\   \left[\Z/p^{v(i)+1}\Z\right](x_i) .$$
%
%
We fix in this proof $\deg x=1,$ so the $i$-th degree component of $I_p$ is given by the ideal $p \tilde{I}_i \subset \Z,$ where
$$\tilde{I}_i\ :=\ \left(p^{v(1)} \begin{pmatrix} i \\ 1  \end{pmatrix}, \dots , p^{v(i-1)} \begin{pmatrix} i \\ i-1  \end{pmatrix}, \ p^{v(i)}\right).$$ 
It is easy to see that actually $\tilde{I}_i=(p^{v(i)}).$

Now we show that the $x_{p^i}$'s are generators of $\Gamma_p[x]$ as a ring. 
We will make use of the following well known (and easy) lemma. 

\begin{lem} \label{factorial} Let \ $n=\sum_j\ n_j\ p^j$ \ 
be the $p$-adic expansions of $n$. Then
$$v(n!)\ =\  \frac{n-\sum_j\ n_j}{p-1}\ =\ \sum_j\ n_j\ \frac{p^j-1}{p-1}.$$

\vspace{-\baselineskip}
\qed
\end{lem}

Let $n:=\sum_{j=0}^k\ n_j\ p^j$ be the $p$-adic expansion of $n.$ 
It 
follows from (\ref{dividedpoly}) that
\begin{equation}\label{p-adic}
\prod_{j=0}^k\ (x_{p^j})^{n_j}\ =\ \frac{n!}{(p!)^{n_1}(p^2!)^{n_2}\dots
  (p^k!)^{n_k}}\ x_{n}. 
\end{equation}
From lemma \ref{factorial} the coefficient of $x_n$ in the latter
expression is not divisible by $p.$  
This clearly shows that the $x_{p^i}$'s are generators. 

Again by (\ref{dividedpoly}) we have
$$(x_{p^i})^p\ =\ \frac{p^{i+1}!}{(p^{i}!)^p}\ x_{p^{i+1}}.$$
By lemma \ref{factorial} the coefficient of \ $x_{p^{i+1}}$ \ in the
latter expression is divisible by  $p$ to the first power. Therefore,
up to multiplying each $x_{p^i}$ by a number which is a unit in
$\Z_{p^{v(i)+1}},$  we get the relations 
$$(x_{p^i})^p\ =\ p\ x_{p^{i+1}}, \quad i\geq 1.$$  
Notice that these relations, together with $px_1=0,$ imply by
induction that  $p^{v(i)+1}\ x_{p^i}=0.$ Therefore one has a well
defined surjection between $\tilde{\Gamma}_p[x]:=\Z[\xi_1,\xi_p,
  \xi_{p^2}, \xi_{p^3}, \ldots]/J_p$ and $\Gamma_p[x].$ 

In $\tilde{\Gamma}_p[x]$ there exist a unique normal form, which derives
from the unique $p$-adic expansion of a natural number $n$: each class
is uniquely represented as a combination of monomials
$\prod_{j=0}^k\ (\xi_{p^j})^{n_j}.$  Such monomials map to
(\ref{p-adic}),  an invertible multiple of $x_n,$ as we have seen
before. 

This concludes the proof of the proposition \ref{prop:algebra}.
\end{proof}
\begin{rmk}
Note that if \ $\Gamma[x]$ \ is graded with $\deg x = k$, then also
$\Z[\xi_1, \xi_{p}, \xi_{p^2}, \ldots]/J_p$ is graded, with $\deg
\xi_{p^i} = k p^i$.
\end{rmk}
\begin{prop}
Let $\Z_{(p)}$ be the local ring obtained by inverting numbers prime
to $p$ and let  
$\Gamma_{\Z_{(p)}}[x]$ be the divided polynomial algebra over $\Z_{(p)}.$ 
 One has an isomorphism:
 $$\Gamma_p[x]\ \cong\ \Gamma_{\Z_{(p)}}[x]/(px).$$
\end{prop}
\begin{proof} The ideal $(px)$ is 
graded (here we also set $\deg x=1$)
with degree $i$ 
%
%
component given by the ideal 
$$(p i) = (p^{v(i)+1}) \subset \ \Z_{(p)}. $$ 
This statement follows directly from the natural isomorphism
$$\Z_{(p)}/(p^k)\ \cong\ \Z_{p^k}.$$

\vspace{-\baselineskip}
\end{proof}

We can naturally define a divided polynomial ring in several variables 
as follows:
$$
\Gamma[x,x',x'', \ldots] := \Gamma[x] \otimes_\Z \Gamma[x'] \otimes_\Z
\Gamma[x'']\otimes_\Z \cdots 
$$
with the ring structure induced as subring of
$\Q[x,x',x'',\ldots]$. In a similar way we have 
$$
\Gamma_p[x,x',x'', \ldots] := \Gamma_p[x] \otimes_\Z \Gamma_p[x']\otimes_\Z
\Gamma_p[x''] \otimes_\Z \cdots. 
$$

In the following we need only to consider the torsion part of $\Gamma_p,$
which is that  generated  in degree greater than $0$. So, we define the submodule 
$$ {\Gamma}^+_p[x, x', x'', \ldots] := {\Gamma}_p[x, x', x'', \ldots]_{\deg > 0}.$$

For every prime $p$ and for $k\geq 1$ define polynomials
$$\mathcal P_{p^k}\ :=\ xy(x^{p^k-1}-y^{p^k-1}).$$ Set 
$$\mathcal P_p:=\mathcal P_{p^1},\ \mathcal Q_p:=\mathcal P_{p^2}/\mathcal P_p=\ \sum_{h=0}^p\ (x^{p-1})^{p-h}(y^{p-1})^{h}\ \in\Z[x,y].$$  

\begin{rmk} There is a natural action of $\SL$ on $(\CP^{\infty})^2$  which induces on the cohomology ring $$H^*((\CP^{\infty})^2;\Z)\ =\ \Z[x,y]$$
the same action defined in section \ref{general} (see \cite{fty}). 
For coherence with this geometrical action, from now on we fix on $\Z[x,y]$ the grading $\deg x = \deg y = 2.$ 
Then we have  $$\deg \mathcal P_p = 2(p+1), \deg \mathcal Q_p = 2p(p-1).$$ 
\end{rmk}
%
%
%
%
%
%
\begin{rmk}
Let $\mathrm{Diff}_+ T^2$ the group of all orientation preserving diffeomorphisms 
of the $2$-dimensional torus $T^2$ and 
let $\mathrm{Diff}_0 T^2$ be the connected component of the identity. 
As showed in \cite{earle-eels} the inclusion $T^2 \subset \mathrm{Diff}_0 T^2$ is
an homotopy equivalence. We recall from \cite{fty} the exact sequence of groups
$$
1 \to \mathrm{Diff}_0 T^2 \to \mathrm{Diff}_+ T^2 \to \SL \to 1
$$
that induces a Serre spectral sequence
$$
E_2^{i,j}= H^i(\SL; H^j((\CP^\infty)^2;\Z)) \Rightarrow H^{i+j}(\mathrm{Diff}_+ T^2; \Z).
$$
The spectral sequence  above collapses with coefficients $\Z[ 1/6 ]$. 
The total space 
$\mathrm{Diff}_+ T^2$ is homotopy equivalent to the Borel construction
$$
E \times_{SL_2(\Z)} (\C\mathbb P^{\infty}) ^2
$$
where $E$ is a contractible space with free $\SL$-action.

Our cohomological computations turn out to give all the $p$-torsion in $H^{*}(\mathrm{BDiff}_+ T^2; \Z)$ 
for a prime $p > 3$.
\end{rmk}
%
%
%
%
%
%
%



Let $G$ be a subgroup of $SL_2(\Z).$ For any prime $p,$ we define the
Poincar\'e series 
\begin{equation}\label{poincarep}
P^i_{G,p}(t) := \sum_{n=0}^\infty \dim_{\F_p} (H^i(G;M_n)\otimes \F_p) t^n
\end{equation}
and for rational coefficients
\begin{equation}\label{poincare0}
P^i_{G,0}(t) := \sum_{n=0}^\infty \dim_{\Q} H^i(G;M_n\otimes \Q)t^n.
\end{equation}

We obtain:
\begin{thm} \label{t:H^1}
Let $M=\oplus_{n\geq 0}\ M_n \simeq H^*((\CP^\infty)^2;\Z)$ be the $\SL$-module defined in the
previous section. Then
\begin{enumerate}
\item $H^0(B_3;M)\ =\  H^0(\SL,M)\ =\ \Z$ concentrated in degree $n=0$;
\item $H^1(B_3;M_n)_{(p)}\ =\ H^1(\SL;M_n)_{(p)}\ =\ \Gamma^+_p[\mathcal P_p, \mathcal Q_p]_{\deg = n}$;
\item $H^1(\SL;M_0) = 0$; $H^1(B_3;M_0) = \Z$; 
\item $FH^1(B_3;M_n) \ =\ FH^1(\SL;M_n)  \ =\  \Z^{f_n}$ for $n >
  0$\end{enumerate} 
where the rank $f_n$ is given by the Poincar\'e series
\begin{equation}
P^1_{B_3,0}(t) 
= \sum_{n=0}^\infty f_n t^n = \frac{t^4(1+t^4-t^{12}+t^{16})}{(1-t^8)(1-t^{12})} 
\end{equation}
and we fix gradings $\deg \mathcal P_p = 2(p+1)$ and $\deg \mathcal Q_p = 2p(p-1)$.
\end{thm}
The proof of part 1 of Theorem \ref{t:H^1} is given in proposition \ref{nointegerinvariants}; 
part 2 is proved in Theorem \ref{ptorsionH1}; 
the proof of part 3 and 4 follows from the results in Section \ref{cohom246};
in particular the Poincar\`e series $P^1_{B_3,0}(t)$ is computed in remark \ref{dimteo1part3}.

\begin{thm} \label{t:series23}
For $i > 1$ the cohomology $H^i(SL_2(\Z),M)$ is $2$-periodic. The free
part is zero and  
there is no $p$-torsion for $p \neq 2,3$. There is no $2^i$-torsion
for $i>2$ and no $3^i$-torsion for $i>1$.  

There is a rank-1 module of $4$-torsion in $H^{2n}(SL_2(\Z),M_{8m})$,
all the others modules have no $4$-torsion. 

Finally the $2$ and $3$-torsion components are determined by the
following Poincar\'e polynomials: 
\begin{eqnarray}
P^{2}_{SL_2(\Z), 2}(t) & = & \frac{1 - t^4 + 2 t^6 - t^8 +
  t^{12}}{(1-t^2)(1+t^6)(1-t^8)} \label{e:2a}\\ 
P^{3}_{SL_2(\Z), 2}(t) & = &
\frac{t^4(2-t^2+t^4+t^6-t^8)}{(1-t^2)(1+t^6)(1-t^8)} \label{e:2b}\\ 
P^{2}_{SL_2(\Z), 3}(t) & = & \frac{1}{1-t^{12}} \label{e:2c}\\
P^{3}_{SL_2(\Z), 3}(t) & = & \frac{t^8}{1-t^{12}}. \label{e:2d} 
\end{eqnarray}
\end{thm}
The proof of Theorem \ref{t:series23} follows from the results in Section \ref{s:sl2modp}; Equations (\ref{e:2a}) 
and (\ref{e:2b}) follow from proposition \ref{p:2series}; Equations (\ref{e:2c}) and (\ref{e:2d}) follow from 
Lemma \ref{l:iso3} and the computations of Section \ref{cohom246}.

We recall (Corollary \ref{bigtor})  that for any $SL_2(\Z)$-module $N$
the free part and part of $p$-torsion in $H^1(SL_2(\Z);N)$ and  
$H^2(B_3;N)$ are the same, with $p \neq 2$ or $3$. We obtain the
following theorems. 

\begin{thm} \label{t:H^2}\ 
The following equalities hold:
\begin{enumerate} 
\item $H^2(B_3;M)_{(p)}\ =\ H^1(B_3;M)_{(p)}  =\ \Gamma^+_p[\mathcal P_p,
  \mathcal Q_p]$ for $p \neq 2,3$; 
\item $FH^2(B_3;M_n) \ =\ FH^1(B_3;M_n)  \ =\  \Z^{f_n}$
\end{enumerate}
where the ranks $f_n$ have been defined in Theorem \ref{t:H^1}. \end{thm}

The proof of the isomorphisms of Theorem \ref{t:H^2} follows from the desctiprion of the spectral
sequence given in Section \ref{s:cohomB3}; the last isomorphism of
part 1 follows from Theorem \ref{ptorsionH1}.

For the $2$ and $3$-torsion components, we find that  the second
cohomology group of $B_3$  differs from the first cohomology group as
follows. We consider the $\Z$-module structure, and notice that  
$\Gamma_p^+[\mathcal Q_p]$ and  $\mathcal P_p\cdot\Gamma_p^+[\mathcal Q_p]$  are
direct summand submodules of  $\Gamma_p^+[\mathcal P_p,\mathcal Q_p],$ for all
$p.$  In the second cohomology group we have modifications of these
submodules.

\begin{thm} \label{t:B_3primes23} 
For the \ $2$ and  $3$-torsion components of the module $H^2(B_3;M)$
one has the following expressions: 
\begin{itemize} 
\item[a)] $H^2(B_3; M)_{(2)} = (\Gamma^+_2[\mathcal P_2, \mathcal Q_2] \oplus
  \Z[\overline{\mathcal Q}_2^2])/\sim$  

Here $\overline{\mathcal Q}_2$ is a new variable of the same degree
($=4$) as $\mathcal Q_2$;  the quotient module is defined by the
relations $\frac{\mathcal Q_2^n}{n!} \sim 2 \overline{\mathcal Q}_2^n$ for
$n$ even and $\frac{\mathcal Q_2^n}{n!} \sim 0$ for $n$ odd. 
\item[b)] $H^2(B_3; M)_{(3)} = (\Gamma^+_3[\mathcal P_3, \mathcal Q_3] \oplus
  \Z[\overline{\mathcal Q}_3] )/ \sim$  

Here $\overline{\mathcal Q}_3$ is a new variable of the same degree
($=12$) as $\mathcal Q_3$; the quotient module is defined by the relations
$\frac{\mathcal Q_3^n}{n!} \sim 3 \overline{\mathcal Q}_3^n$ and
$\mathcal P_3\frac{\mathcal Q_3^n}{n!} \sim 0$. 
\end{itemize}
\end{thm}

The proof of Theorem \ref{t:B_3primes23} follows from the results of Section \ref{s:cohomB3}.

\begin{rmk}
For $p=2$ one has to kill all the submodules generated by elements of
the form $\frac{\mathcal Q_2^n}{n!}$ for odd $n$. All these submodules
are isomorphic to $\Z/2$. For $n$ even the submodule generated by
$\frac{\mathcal Q_2^n}{n!}$, that is isomorphic to $\Z/2^{m+1}$ where $m$
is the greatest power of $2$ that divides $n$, must be replaced by a
submodule isomorphic to $\Z/2^{m+2}$. 

For $p=3$ one has to kill all the submodules generated by elements of
the form $\mathcal P_3\frac{\mathcal Q_3^n}{n!}$, that are isomorphic to
$\Z/3$. Moreover for all $n$ the submodule generated by
$\frac{\mathcal Q_3^n}{n!}$, that is isomorphic to $\Z/3^{m+1}$ where $m$
is the greatest power of $3$ that divides $n$, must be replaced by a
submodule isomorphic to $\Z/3^{m+2}$. 
\end{rmk}

%% file: poly_inv.tex
\section{Polynomial invariants for $SL_2(\Z)$}\label{invariants}

In this section we prove the part of the main theorem (\ref{results}) which concerns $H^1(B_3;M_n)$ $=$ $H^1(\SL;M_n)$ for $n>0$.

A classical result (\cite{dickson}) characterizes the polynomials in $\F_p[x,y]$ which are invariant under the action of $GL_n(\F_p)$ (or of $SL_n(\F_p)).$ We state only the case $n=2,$ which is what we need here.   


\begin{thm}\label{teo:invariants}(\cite{dickson},\cite{steinberg}) For $p$ a prime number,   the algebra of  $SL_2(\F_p)$-invariants in $\F_p[x,y]$ is the polynomial algebra
$$\F_p[\first_p,\second_p],$$
where $\first_p,\ \second_p$ are the polynomials defined in part \ref{results}.
\end{thm}


Of course, the algebra $\F_p[\first_p,\second_p]$ is 
graded. We keep here $\deg\ x=\deg\ y =2,$ so  
the degree is that induced by  $\deg \first_p=2(p+1),$ and $\deg \second_p=2p(p-1).$  

\begin{thm}\label{t:invariants1}(\cite[Th. J]{steinberg}) For  $p$  a
  prime number,  a polynomial $P\in \Ar[x,y]$ is invariant for
  $SL_2(\Ar)$ if and only if     
$$P=\sum_{i=0}^{r-1} p^i\ F_i$$
where $F_i$ is a polynomial in the $p^{r-i-1}-$th powers of the
$\first_p,\ \second_p.$ 
\end{thm}

%
%
%
 For our purposes, we need a slight variation of the statement of this theorem.

\begin{cor}\label{cor:invariants} A polynomial $P\in \Ar[x,y],$ homogeneous 
 of degree $n,$ is invariant under the action of $\SL$ on $\Ar[x,y]$ induced 
 by the action on $M_n,$  if and only if 
$$P=\sum_{i=0}^{r-1} p^i\ F_i$$
where $F_i$ is a polynomial in the $p^{r-i-1}-$th powers of the
$\first_p,\ \second_p.$ 
\end{cor}

\begin{proof} The corollary follows directly from theorem
  (\ref{t:invariants1})   
 and from the surjectivity of  the natural projection $\SL \to SL_2(\Ar)$. 
In order to prove the surjectivity consider a matrix
$$\overline{A} = \begin{pmatrix} \overline{a} & \overline{b} \\
\overline{c} & \overline{d} \end{pmatrix} \in SL_2(Z/p^rZ).$$
We suppose $\overline{c}=0$.
There is a matrix $A$ with integer coefficients 
such that $A$ maps to $\overline{A}$ and we have
$$
A = \begin{pmatrix} a & b \\ c & d \end{pmatrix} $$
with $ad-bc = 1 + kp^r$. Clearly $d$ is not divisible by $p$ since $\overline{d}$ is invertible.
We can choose $c=p^r$ and hence $(c,d) = 1$. Then we choose new values
$a' = a + xp^r$, $b' = b+ yp^r$ and the determinant of the matrix
$$
A' = \begin{pmatrix} a' & b' \\ c & d \end{pmatrix} $$
is $1+kp^r+p^r(xd+yp^r)$. Hence we can find values of $x,y$ such that $\det A' =1$.
%
%
%
%
%
The cases when any of the other entries is zero is analogous and 
the surjectivity follows since 
triangular matrixes with a zero entry generate the group $SL_2(\Z_{p^r})$.
\end{proof}

We also need

\begin{prop}\label{nointegerinvariants} If $n>0,$ there are no
  invariants in $M_n$ under the action of $\SL.$ 
\end{prop}
\begin{proof} One can prove this result directly. However, it easily
  derives from theorem \ref{teo:invariants}. In fact, let
  $P\in\Z[x,y]$ be an invariant, homogeneous of degree $d.$ Then $P$
  must be invariant modulo all primes $p,$ which is impossible by
  theorem $\ref{teo:invariants}:$ the polynomial $P$ should be $0$
  modulo $p$ for every $p>d$. 
\end{proof}

Corollary \ref{cor:invariants} and proposition
\ref{nointegerinvariants} give us, by the Universal Coefficient
Theorem, a complete description of 
$$H^0(\SL;M\otimes\Ar)\ =\  \Ar \oplus Tor(H^1(\SL;M);\Ar)$$
(and similar for $B_3$).  The result can be stated in the following way.

\begin{thm}\label{ptorsionH1} 
For all prime numbers $p$ and all $n>0,$ the $p-$torsion component 
$$H^1(\SL; M_n)_{(p)}\ \cong\ H^1(B_3 ; M_n)_{(p)}$$
is described as follows:
each monomial 
$$\first_p^k\second_p^h,\quad 2 k(p+1)+2 hp(p-1)=n$$
generates a $\Z_{p^{m+1}}-$summand, where
$$p^m\mid gcd(k,h),\ p^{m+1}\nmid gcd(k,h).$$

\vspace{-\baselineskip}
\qed\end{thm}
  

Recall that since we fix $\deg x = \deg y = 2$ we have that the module $M$ is concentrated 
even degree.

The condition stated in the previous theorem over the number $m$ is equivalent to saying that
$m$ is the minimum power of $p$ which is present in the $p$-base expansion of $k$ and $h.$

It is easy to check the equivalence between theorem \ref{ptorsionH1}  and the description of the cohomology given in part 3 by using divided powers.

%% file: cohom246.tex
\section{Cohomology of $\Z_2$, $\Z_4$, $\Z_6$}\label{cohom246}
Recall from part \ref{general}, proposition \ref{presentation}, that $w_2 = (s_1s_2s_1)^2 \in SL_2(\Z)$ generates a group isomorphic to $\Z_2,$ acting on the module $\Z^2$ as multiplication by the matrix $\begin{bmatrix} -1 & 0 \\ 0 & -1 \end{bmatrix}.$

\begin{prop}\label{z2cohom}
For even $n$ we have:
$$
H^i(\Z_2;M_{2n}) = \left\{ \begin{array}{ll} 
M_{2n} & \mbox{ if } i =0 \\
0 & \mbox{ if }  i \mbox { is odd} \\
M_{2n} \otimes \Z_2 & \mbox{ if }  i>0 \mbox{ is even}
\end{array}
 \right.
 $$
 For odd $n$ we have:
$$
H^i(\Z_2;M_{2n}) = 
\left\{ \begin{array}{ll} 
0 & \mbox{ if }  i \mbox { is even} \\
M_{2n} \otimes \Z_2 & \mbox{ if }  i \mbox{ is odd} 
\end{array} \right. 
$$ 
For rational coefficients we have:
$$H^i(\Z_2;M_n \otimes \Q) = \left\{ \begin{array}{ll} 
M_n \otimes \Q & \mbox{ if } i =0 \\
0 & \mbox{ otherwise. }
\end{array}
 \right.
$$
\end{prop}

\begin{proof}  We use the standard resolution for a cyclic group $\Z_m$ (see \cite{brown}): 
\begin{equation}\label{cyclicres}
\begin{CD}
N @>T-1>>  N @>T^{m-1}+\dots+1>> N @>T-1>> \dots 
\end{CD}
\end{equation}
where $N$ is any $\Z_m-$module with  action given by a  $T$ such that $T^m=1.$ 

In the present case the sequence becomes:
\begin{equation}\label{cyclicresmod2}
\begin{CD}
M_n @>T-1>>  M_n @>T+1>> M_n @>T-1>> \dots 
\end{CD}
\end{equation}
Notice that $T$ acts as $Id$ in degree $2n$, $n$ even
, while it acts as $-Id$ in degree $2n$ with $n$ odd
. Then everything follows trivially from sequence (\ref{cyclicresmod2}). 
\end{proof}

The following proposition describes the ring structure of the \ $\Z_2$-invariants. It directly follows  from proposition \ref{z2cohom}.
\begin{prop}\label{z2invariants}
The module of invariants $H^0(\Z_2; M)$ is isomorphic to the polynomial ring
$$
\Z[a_2,b_2, c_2]/(a_2^2b_2^2 =c_2^2)
$$
under the correspondence $a_2 = x^2$, $b_2=y^2$, $c_2 = xy$.

The module $M\otimes \Z_2$ is $\Z_2$-invariant. \qed
\end{prop}

\begin{prop}\label{poincarez2}
The Poincar\`e series for $H^i(\Z_2, M)$
are given by 
\begin{eqnarray*}
P^{0}_{\Z_2, p}(t) & = & \frac{(1+t^4)}{(1-t^4)^2} 
\\
P^{2i}_{\Z_2,2}(t) & = & \frac{(1+t^4)}{(1-t^4)^2} \\
P^{2i+1}_{\Z_2,2}(t) & = & \frac{2t^2}{(1-t^4)^2}
\end{eqnarray*} 
\end{prop}
\begin{proof} By proposition (\ref{z2cohom}) all we have to consider are the series of the dimensions of the polynomials of degree $2n$
with $n$ even and odd respectively.  These are exactly the series written here.
\end{proof}

Recall (prop \ref{presentation}) that the element $w_4 = s_1s_2s_1 \in SL_2(\Z)$ generates a group isomorphic to $\Z_4$ and the action on $M$  is given by  $\left\{\begin{array}{lr} x \mapsto & -y \\ y \mapsto & x \end{array} \right.$. 

The following proposition is quite trivial to verify.

\begin{prop} \label{prop:Z_4}
The group $M$, as a $\Z_4$ module, is the direct sum of modules of the following kinds:
\begin{itemize}
\item[a)]  $I_1$ \hspace {0.1cm} isomorphic to a submodule linearly generated by the monomials $x^{2i}y^{2j}$, $x^{2j}y^{2i}$ for fixed $i \neq j$;
\item[b)] $I_2$ \hspace {0.1cm} isomorphic to a submodule linearly generated by the monomial $x^{2i}y^{2i}$ for fixed $i$;
\item[c)] $I_3$ \hspace {0.1cm} isomorphic to a submodule linearly generated by the  monomial  $x^{2i+1}y^{2i+1}$ for fixed $i$;
\item[d)] $I_4$ \hspace {0.1cm} isomorphic to a submodule linearly generated by the monomials $x^{2i+1}y^{2j+1}$, $x^{2j+1}y^{2i+1}$ for fixed $i  \neq j$;
\item[e)] $I_5$ \hspace {0.1cm} isomorphic to a submodule linearly generated by the monomials $x^{2i+1}y^{2j}$, $x^{2j}y^{2i+1}$ for fixed $i,  \ j$. \qed
\end{itemize}
\end{prop}

\begin{prop} \label{prop:tabz4}
The cohomology $H^i(\Z_4; I_j)$ is $2$-periodic for $i\geq1$ and is given by the following table:

\begin{center}
\begin{tabular}{|c|c|c|c|c|c|}
\hline
& $I_1$ & $I_2$ & $I_3$ & $I_4$ & $I_5$ \\
\hline
$H^0$ & $\Z$ & $\Z$ & $0$ & $\Z$ & $0$ \\
$H^1$ & $0$ & $0$ & $\Z_2$ & $0$ & $\Z_2$ \\
$H^2$ & $\Z_2$ & $\Z_4$ & $0$ & $\Z_2$ & $0$ \\
\hline
\end{tabular}\end{center}
Referring to the notation of \ref{prop:Z_4} for the generators of the modules, we give the following representatives for the generators of the cohomology:
\begin{center}
\begin{tabular}{|c|c|c|c|c|c|}
\hline
& $I_1$ & $I_2$ & $I_3$ & $I_4$ & $I_5$ \\
\hline
$H^0$ & $x^{2i}y^{2j} + x^{2j}y^{2i}$ & $(xy)^{2i}$ & $0$ & $x^{2i+1}y^{2j+1} - x^{2j+1}y^{2i+1}$ & $0$ \\
$H^1$ & $0$ & $0$ & $(xy)^{2i+1}$ & $0$ & $x^{2i+1}y^{2j}$ \\
$H^2$ & $x^{2i}y^{2j} + x^{2j}y^{2i}$ & $(xy)^{2i}$ & $0$ & $x^{2i+1}y^{2j+1} - x^{2j+1}y^{2i+1}$ & $0$ \\
\hline
\end{tabular}\end{center} 

\end{prop}
\begin{proof} It derives easily from the above standard $2$-periodic resolution for cyclic groups (\ref{cyclicres}), applied to the modules $I_j$'s.
\end{proof}

\begin{prop}\label{z4invariants}
The module of invariants $H^0(\Z_4, M)$ is isomorphic to the polynomial ring
\begin{equation}\label{prop:40}
\Z[d_2, e_4, f_4]/(f_4^2 = (d_2^2 - 4 e_4)e_4)
\end{equation}
under the correspondence $d_2= x^2 + y^2$, $e_4 = x^2y^2$, $f_4= x^3y-xy^3$.

The module of $\mod 2$-invariants $H^0(\Z_4, M\otimes \Z_2)$ is isomorphic  to the polynomial ring
\begin{equation}\label{prop:42}
\Z[\sigma_1,\sigma_2]
\end{equation}
under the correspondence $\sigma_1= x+y$, $\sigma_2=xy$. 
\end{prop}

\begin{proof} In the mod 2-case, notice that $\Z_4$ acts exchanging the variables $x,\ y,$ so the invariants are all the symmetric polynomials.
 
 In the integral case, first notice that  the given polynomials are clearly invariants for $\Z_4.$ 
 
 Second, each class in the ring (\ref{prop:40}) has a standard representative of the form
\begin{equation}\label{normalform4} P(d_2,e_4)+Q(d_2,e_4)f_4.
\end{equation}
A polynomial of this form is zero iff $P$ and $Q$ both vanish, because the variables appear at an odd power in $f_4.$ But notice that $d_2,$ $e_4$  are the standard generators for the symmetric polynomials in $x^2,\ y^2,$ so they are algebraically independent. This means that $P$ and $Q$ are the zero polynomials. Therefore all the relations among $d_2,\ e_4,\ f_4$ derive from the give one and the form (\ref{normalform4}) is unique.

Finally, we have to show that the given polynomials generate all the invariants. It is sufficient to see that the polynomials given in the first row of the generators-table in proposition (\ref{prop:tabz4}) are generated.  

The polynomial appearing in the second column is a power of $e_4.$ The one in the first column, after dividing by the maximum possible power of $e_4,$ becomes a polynomial of the form 
$$x^{2i}+y^{2i}.$$
This is a symmetric polynomial in $x^2,\ y^2$ so it is generated by $d_2,\ e_4.$

The polynomial appearing in the fourth column, after dividing by the maximum possible power of $e_4,$ becomes of the form (up to sign)
$$x^{2i+1}y-xy^{2i+1},\ i>0.$$
This latter is divisible by $f_4,$ with quotient a symmetric polynomial in $x^2,\ y^2.$ 
\end{proof}

\begin{prop}\label{poincarez4}
With respect to the gradation of $M$ the Poincar\`e series for $H^i(\Z_4, M)$
are given by 
\begin{eqnarray*}
P^{2i}_{\Z_4, p}(t) & = & \frac{1+t^8}{(1-t^4)(1-t^8)}  \mbox{ for } p
= 2 \mbox{ or } i = 0 
\\
P^{2i+1}_{\Z_4,2}(t) & = & \frac{t^2+t^4+t^6-t^8}{(1-t^4)(1-t^8)}
\end{eqnarray*} 
\end{prop}

\begin{proof} By looking at the generators-table in proposition \ref{prop:tabz4}, we see that in odd dimension 
one has terms
 $$\frac{t^4}{(1-t^8)}\text{ , }\ \frac{t^2}{(1-t^4)^2}$$
 coming from the generators for  $I_3,\ I_5$ \ respectively. Such terms add up to the expression for $P^{2i+1}_{\Z_4,2}(t)$ given in the statement.

Similarly, in even dimension we have three terms
$$ \frac{t^4}{(1-t^8)(1-t^4)} \text{ , }\ \frac{1}{1-t^8} \text{ , }\ \frac{t^8}{(1-t^8)(1-t^4)}$$
coming from the generators for $I_1,\ I_2,\ I_4$ respectively, which add up to the given expression for 
$P^{2i}_{\Z_4, p}(t).$
\end{proof} 

The cohomology of $\Z_4$ (so also $H^*(SL_2(\Z);M)$)  contains $4$-torsion, as it follows from  proposition \ref{prop:tabz4}.
We define in general the polynomial 
\begin{equation}\label{poincare4}
P^i_{G,4}(t) := \sum_{n=0}^\infty \dim_{\F_2} (2H^i(G;M_n)) t^n.
\end{equation}
We have immediately from proposition \ref{prop:tabz4}
\begin{prop}
$$
P^1_{\Z_4,4}(t) = \frac{1}{1-t^8}. 
$$ 

\vspace{-\baselineskip}
\qed
\end{prop}

Recall (proposition \ref{presentation}) that the element $w_6 = s_1s_2 \in SL_2(\Z)$ generates a group isomorphic to $\Z_6,$  where the action over  $M$ is  \ $T:\left\{\begin{array}{lr} x \mapsto & -y \\ y \mapsto & x + y \end{array} \right.$.

\begin{prop}\label{prop:invariants}
The module of invariants $H^0(\Z_6, M)$ is isomorphic to the polynomial ring
\begin{equation}\label{prop:60}
\Z[p_2, q_6, r_6 ]/(r_6^2= q_6(p_2^3-13q_6-5r_6))
\end{equation}
via the correspondence $p_2= x^2 + xy+  y^2$, 
$q_6 = x^2y^2(x+y)^2 $,
$r_6 = x^5y -5x^3y^3 -5 x^2y^4 -xy^5$.

The module of $\mod 2$-invariants $H^0(\Z_6, M\otimes \Z_2)$ is isomorphic to the polynomial ring
\begin{equation}\label{prop:62}
\Z_2[s_2, t_3, u_3]/(u_3^2=  s_2^3 + t_3^2 + t_3u_3 )
\end{equation}
with the correspondence $s_2 = x^2 + xy + y^2$, $t_3 = xy(x+y)$, $u_3=x^3 + x^2y+y^3$.

The module of $\mod 3$-invariants $H^0(\Z_6, M\otimes \Z_3)$ is isomorphic to the polynomial ring
\begin{equation}\label{prop:63}
\Z_3[v_2, w_4, z_6]/(w_4^2=v_2z_6)
\end{equation}
with the correspondence $v_2= (x-y)^2$, $w_4=xy(x+y)(x-y)$, $z_6= x^2y^2(x+y)^2$.
\end{prop}
\begin{proof} It is an easy check that $p_2,\ q_6$ and $r_6=xy(x+y)(x^3-x^2y-4xy^2-y^3)$ are invariant. Moreover it is easy to check the relation $r_6^2= q_6(p_2^3-13q_6-5r_6)$ stated in the theorem. 

Modulo this relation, each polynomial can be reduced to one of the form
\begin{equation}\label{normalform} 
P(p_2,q_6)+Q(p_2,q_6)r_6. 
\end{equation}
Notice that $p_2$ and $q_6$ are symmetric polynomials, while the last factor of $r_6$ is not. This implies  that a polynomial of type (\ref{normalform}) vanishes iff $P$ and $Q$ both vanish.  

An homogeneous relation between $p_2$ and  $q_6$ would be of the form
\begin{eqnarray*}
 0 & = & \sum_{k\geq 0}\ a_k\ p_2^{n-3k}q_6^k\ = \\
  & = &  \sum_{k\geq 0}\ a_k\ (x^{2n-2k}y^{2k}+ \text{terms containing $x$ in lower degree)}.
\end{eqnarray*}
Since each term in the sum contains $x$ to a maximum power which is decreasing with $k,$ all the $a_k$'s must vanish. Therefore, all relations follow from the given one and the form (\ref{normalform}) is unique.

We now have to show that the given polynomials generate.  First notice that the ring  in (\ref{prop:60}) is a free 
$\Z$-module. The Poincar\'e  series reproducing its ranks is given, by the normal form (\ref{normalform}), by
\begin{equation}\label{ranks}  \frac{1+t^{12}}{(1-t^4)(1-t^{12})}.
\end{equation}
Let $\omega$ be a primitive $6$th-root of the unity. Let us extend the action of $\Z_6$ over the extension 
$$\Q(\omega)[x,y]\ \supset \Z[x,y]$$
where $\Q(\omega)$ is the cyclotomic field generated by $\omega.$  One can diagonalize the action as follows: one sees immediately that the two polynomials
$$p(x,y):=\overline{\omega} x+y,\ \overline{p}(x,y)={\omega}x+y$$
are eigenvectors with eigenvalues ${\omega}$, $\overline{\omega}$ respectively.  In degree $2n,$ one has eigenvectors
$$p^a\overline{p}^b,\ a+b=n.$$
with eigenvalue $\omega^a\overline{\omega}^b.$ It follows that the invariants are generated by all the monomialls $p^a\overline{p}^b$ such that $a\equiv b\ \text{(mod\ $6$)}.$ It is quite easy to check that the Poincar\'e series of the dimensions of the vector space of the invariants is given by the same series (\ref{ranks}). Since such dimensions do not dependent on the extension, it follows that the given polynomials $p_2,\ q_6,$ $r_6$ generate the invariants over the rational numbers. So, in each degree,  they generate over $\Z$ a lattice of the same rank as that of all the invariants, and we have to exclude that they generate a sublattice.

An homogeneous integral invariant of degree $4n$
is a rational combination (keeping into account (\ref{normalform})) of the form
 
\begin{eqnarray*}
t(x,y) & = & \sum_{a+3b=n} \  \lambda_{a,b}\ p_2^{a}q_6^{b}+ \sum_{a+3b=n-3} \  \mu_{a,b}\ p_2^{a}q_6^{b}r_6\ = \\
& = & \sum_{a+3b=n} \  \lambda_{a,b}\ \text{(}x^{2a+4b}y^{2b}+\text{terms containing $x$ in lower degree)}\ + \\
& + & \sum_{a+3b=n-3} \  \mu_{a,b}\  \text{(}x^{2a+4b+5}y^{2b+1}+\text{terms containing $x$ in lower degree)}  
\end{eqnarray*}
with $\lambda_{a,b},\ \mu_{a,b}\in\Q$. Since the maximum powers of $x$ which are contained into each summand are pairwise different, and they appear with coefficient $1,$ all the $\lambda$'s and $\mu$'s must be integer. This concludes the first part of the theorem on the integral case.

The proof in case of $\Z_2$-coefficients is similar. First one verifies that the given polynomials $s_2,\ t_3,$ $u_3$ are invariants and the given relation holds.  Each class in the ring (\ref{prop:62}) has a unique representative of the kind
$$P(s_2,t_3)+Q(s_2,t_3)u_3.$$
In fact, since $s_2,\ t_3$ are symmetric while $u_3$ is not, an argument as in the integral case proves that all the relations among the given polynomials derive from that in (\ref{prop:62}).
 
 Next, we consider the action over the extension
$$\F_4[x,y]\ \supset\ \Z_2[x,y]$$
where we identify the field $F_4$ with $\Z_2[x]/(x^2+x+1).$ Let $a:=[x]\in\F_4.$ There are eigenvectors
$$p(x,y):= (a+1)x+y,\ \overline{p}(x,y):=ax+y$$
with eigenvalues $a,\ a+1$ respectively. In degree $2n$
, there are eigenvectors $p^k\overline{p}^h,$ $k+h=n,$ with eigenvalues $a^k(a+1)^h.$ Similar to the integral case, we deduce that the invariants correspond to the case 
$$k\equiv h\ \text{(mod\ $3$)}.$$
So, the Poincar\'e series of the dimensions of the invariants, gradated by the degree, is
\begin{equation}\label{ranks2}
\frac{1+t^6}{(1-t^4)(1-t^6)}
\end{equation}
which coincides with the Poincare' series of the ring (\ref{prop:62}).

In case of $\Z_3$-coefficients, again one verifies that the given polynomials are invariant and the given relation holds.  As in the previous cases, one sees that each class in the ring (\ref{prop:63}) has a unique representative of the kind
\begin{equation}\label{normalform6} 
P(v_2,z_6)+Q(v_2,z_6)w_4.
\end{equation}

Notice that, as in the integral case, there are no invariants in degree $2n$, $n$ odd
, since $T^3=-Id.$ On the contrary, the operator does not diagonalize in this case. In degree $2$
, with respect to $p=x-y,$ $q=x$ the operator $T$ is given by the matrix 
$\begin{bmatrix}  -1& 1 \\  0 & -1 \end{bmatrix}.$ In degree $4n,$
with respect to the basis $p^{2n-k}q^k,$ $0\leq k\leq 2n,$  $T$ is given by an upper triangular matrix with the only eigenvalue $1$.  Notice that in degree $4n$
we have  $T^3=Id,$ so
\begin{equation}\label{orderofT} (T-Id)^3\ =\ T^3-Id\ =\ 0.\end{equation}
Therefore, the Jordan blocks of $T$ have dimension at most $3.$

We need the following lemma. Denote for brevity by $M'_d:=M_d\otimes \Z_3.$

\begin{lem}\label{lem:invariants}
Let $P$ be an invariant polynomial  of degree \  $d$ \ (so $P\in\ ker(T-Id):M'_d\to M'_d$). Assume that  $P\in Im(T-Id)^{r-1},$\ $P\not\in Im(T-Id)^r$. Then, for every invariant polynomial $Q$ of degree $l$ the product $QP\in M'_{l+d}$ is an invariant of degree $l + d$ such that  $QP\in Im(T-Id)^{r-1}.$
\end{lem}
\begin{proof} Let $R$ be a polynomial such that $(T-Id)^{r-1}(R)=P.$ Then, since $Q$ is invariant we have  
\[
(T-Id)^{r-1}(QR)=\ Q(T-Id)^{r-1}(R)\ =\ QP. 
\]

\vspace{-\baselineskip}
\end{proof}
 
 In degree $4$, we have 
$$(T-Id)^2(-x^2)\ =\ v_2$$
so here there is a single Jordan block of dimension $3$ for $T-Id.$   It follows from lemma \ref{lem:invariants} and from equation (\ref{orderofT}) that
each polynomial of the form $v_2u,$ $u$ invariant, satisfies:
$$v_2u\in Im(T-Id)^{2},\ v_2u\not\in Im(T-Id)^3.$$

In degree $8$, we have 
$$(T-Id)(-x^3y-x^2y^2)\ =\ w_4, \ w_4\not\in Im(T-Id)^2,$$
so $v_2^2,\ w_4$ belong to a Jordan base where $v_2$ is contained into a $3$-block and $w_4$ is contained into a $2$-block.  It also derives from lemma \ref{lem:invariants} that an invariant of the form $w_4u,$ $u$ invariant, satisfies $w_4u\in Im(T-Id).$

In degree $12$,
we have $z_6\ \not\in Im(T-I),$ so we have 
two $3$-dimensional blocks containing $v_2^3$ and $v_2w_4$ respectively, and
a $1$-dimensional block containing $z_6.$  

The Poincar\'e series for the ring (\ref{prop:63}) is given by
\begin{equation}\label{ranks6}  \frac{1+t^8}{(1-t^8)(1-t^{12})}. 
\end{equation}
In $M_n$ all monomials in the $v_2,\ w_4,\ z_6$ (reduced in normal form, with the exponent of
$w_4$ not bigger than $1$) are linearly independent. They can be taken as part of a Jordan basis for $T,$ where all monomials containing $v_2$ belong to a Jordan block of dimension exactly $3.$ All monomials containing $w_4$ and not $v_2$ belong to a Jordan block of dimension $\geq 2,$ and monomials containing only 
$z_6$
belong to Jordan blocks of dimension $\geq 1.$ Let  $\mu_n$  be the sum of the dimensions of the Jordan blocks which contain all monomials of degree $n$ in the $v_2, w_4,$ 
$z_6,$
where the exponent of 
$z_6$ 
is $\leq 1.$  Let $\mu(t)\ :=\ \sum_n\ \mu_n t^n$ be the associated series. By the above arguments we have that the series $\mu$ is bigger (term by term) than
$$3\left[\frac{1+t^8}{(1-t^4)(1-t^{12})}\ -\ \frac{1+t^8}{(1-t^{12})}\right] \ +\ 2\frac{t^8}{(1-t^{12})}\ +\ \frac{1}{(1-t^{12})}\ = \frac{1+t^4}{(1-t^4)^2}$$
but this last series is exactly
$$\frac{d}{d\tau}\frac{\tau}{1-\tau^2}\ =\ \sum_{k\geq 0}\ (2k+1)\tau^{2k} \quad \mbox{ with }\tau = t^2$$
which coincides with the series of the dimensions of even-degree polynomials. Therefore the two series coincide, which means that $v_2, w_4,$ $z_6$ generate the invariants. 

This concludes the proof of the proposition. 
\end{proof}

\begin{prop}\label{poincarez6}
With respect to the gradation of $M$ the Poincar\`e series for $H^i(\Z_6, M)$
are given by 
\begin{eqnarray*}
P^{2i}_{\Z_6, p}(t) & = & \frac{1+t^{12}}{(1-t^4)(1-t^{12})} \quad \mbox{
  for } p = 2 \mbox{ or }\  i = 0,\mbox{ any }p 
\\
P^{2i}_{\Z_6, 3}(t) & = & \frac{1}{1-t^{12}} \quad \mbox{ for } \  i > 0 \\
P^{2i+1}_{\Z_6,2}(t) & = & \frac{2t^6}{(1-t^4)(1-t^{12})}\\
P^{2i+1}_{\Z_6,3}(t) & = & \frac{t^8}{1-t^{12}}
\end{eqnarray*} 
\end{prop}
\begin{proof}
The first formula for $p=0$ is just the Poncar\'e series of (\ref{prop:60}) (see (\ref{ranks})). In general we use the resolution (\ref{cyclicres}) and the Universal Coefficient Theorem in the form
\begin{equation}\label{universal}
\tilde{P}_{\Z_6,p}^{i}(t)\ =\ P_{\Z_6,p}^{i}(t)\ +\ P_{\Z_6,p}^{i+1}(t) 
\end{equation}
where we set
$$\tilde{P}_{\Z_6,p}^{i}(t)\ =\ \sum_{n\geq 0}\ dim_{\Z_p}\ H^i(\Z_6;M_n\otimes \Z_p)\ t^n.$$
Let $T':=T^5+\dots+T+1.$

In characteristic $2$ one has $T^3=Id$ and $T'=0$ in any degree.  It follows  
$$H^0(\Z_6;M\otimes \Z_2)= H^{2i}(\Z_6;M\otimes \Z_2)\ \equiv  H^{2i-1}(\Z_6;M\otimes \Z_2),\  i>0.$$
Therefore $\tilde{P}^0_{\Z_6,2}=\tilde{P}^1_{\Z_6,2}$ and by (\ref{universal}) $P^0=P^{2i},\ i>0,$  which gives the wanted expression for $P^{2i}_{\Z_6,2}.$
 
By (\ref{universal}) and by the knowledge of the Poincare' series of the $H^0$'s, we recursively find, using (\ref{ranks}):
 $$P^1_{\Z_6,2}(t)\ =\ \frac{1+t^6}{(1-t^4)(1-t^6)}\ -\ \frac{1+t^{12}}{(1-t^4)(1-t^{12})}$$
which gives the written expression for $P^{2i+1}_{\Z_6,2},$ concluding the case of characteristic $2$. 
 
In a similar way we have in characteristic $3$  (using (\ref{ranks6})):
\begin{eqnarray*}
P^1_{\Z_6,3}(t) & = &\frac{1+t^8}{(1-t^4)(1-t^{12})}\ -\ \frac{1+t^{12}}{(1-t^4)(1-t^{12})} = \\ 
& =& \frac{t^8}{(1-t^{12})}\ = P^{2i+1}_{\Z_6,3}(t),\quad i>0
\end{eqnarray*}
which gives the above formula.

We now calculate $\tilde{P}^{2i+1}_{\Z_6,3}.$ 
It derives from the discussion given in the proof of proposition \ref{prop:invariants} that $T$ decompose into Jordan blocks of dimension $\leq 3,$ where the $3$-blocks correspond to the monomials containing a $v_2$ generator. One sees immediately that   
$$ \mathrm{rank}(T')\ =\  \mbox{number of $3$-block
s in $T$}$$
and therefore  (by (\ref{ranks6})) 
$$\tilde{P}^{2i}_{\Z_6,3}(t)\ =\ \tilde{P}^{2i-1 }_{\Z_6,3}(t)\ =\ \frac{1+t^8}{1-t^{12}},\ \quad i>0.$$
Therefore, by (\ref{universal})
$$P^{2i}_{\Z_6,3}(t)\ =\ \frac{1+t^8}{1-t^{12}}\ -\ \frac{t^8}{1-t^{12}}\ =\ \frac{1}{1-t^{12}}$$
which is the given formula. 
\end{proof}

 \begin{rmk}\label{dimteo1part3} 
 As a corollary of the given formulas we obtain part $3$ of theorem \ref{t:H^1}.
In fact, by the Mayer-Vietoris sequence  (\ref{mayervietoris}) and by propositions \ref{poincarez4}, \ref{poincarez6}, \ref{poincarez2}  we have
\begin{eqnarray*}
P^1_{\SL,0}(t) & = & 1\ -\ (P^0_{\Z_4,0}(t)+P^0_{\Z_6,0}(t))\ + P^0_{\Z_2,0}(t)\ = \\
& = & 1\ -\ \left(\frac{1+t^8}{(1-t^4)(1-t^8)}\ +\ \frac{1+t^{12}}{(1-t^4)(1-t^{12}) } \right) + \frac{1+t^4}{(1-t^4)^2}
\end{eqnarray*}
which reduces to the expression given in theorem \ref{t:H^1}. \qed
\end{rmk}

%% file: sl2modp.tex
\section{Cohomology of $SL_2(\Z) \mod p$} \label{s:sl2modp}

We use the results of the previous sections to compute the modules $$H^i(SL_2(\Z);M \otimes \Z_p)$$ 
and the module of $p$-torsion  $H^i(SL_2(\Z);M)_{(p)}$ for any prime $p$. 

In dimension $0$ we have that $H^0(SL_2(\Z);M)$ is isomorphic to $\Z$ concentrated in degree $0$.
The module $$H^0(SL_2(\Z);M \otimes \Z_p)$$ has already been
described in section \ref{invariants} as well as $H^1(SL_2(\Z);M)_{(p)}$ 
(Theorem \ref{ptorsionH1}).  

We can understand the $p$-torsion in $H^1(SL_2(\Z);M)$ by means of the results cited in section \ref{invariants}
for Dickson invariants.
In fact the Universal Coefficients Theorem gives us the following isomorphism:
\begin{equation*}
H^1(SL_2(\Z);M)_{(p)} = \widetilde{\Lambda}_p[\first_p, \second_p].
\end{equation*}
We want to remark that the previus equation gives an isomorphism only at the level of $\Z$-modules. The Poincar\'e series for the previous module is:
$$
P^1_{SL_2(\Z),p} (t) = \frac{t^{2(p+1)} + t^{2p(p-1)} - t^{2(p^2+1)}}{(1-t^{2(p+1)})(1-t^{2p(p-1)})}.
$$
Moreover 
from the decomposition
$$
SL_2(\Z) = \Z_4 *_{\Z_2} \Z_6
$$
we have the associated Mayer-Vietoris cohomology long exact sequence given in equation (\ref{mayervietoris}).

Since $H^i(\Z_4;M)_{(p)} = H^i(\Z_2;M)_{(p)} = 0$ for a prime $p \neq 2$ and $i \geq 1$ we have 
\begin{lem} \label{l:iso3}
For $i \geq 2$, $H^i(SL_2(\Z;M \otimes \Z_3) = H^i(\Z_6; M \otimes \Z_3)$ and we have the isomorphism
\begin{equation*} 
H^i(SL_2(\Z);M)_{(3)} = H^i(\Z_6;M)_{(3)} \; \; \; \; \; i \geq 2. 
\end{equation*}
For any prime $p \neq 2,3$ and for $i \geq 2$ we have $H^i(SL_2(\Z;M \otimes \Z_p) = 
0$ and $H^i(SL_2(\Z);M)_{(p)} = 0$. \qed
\end{lem}

For $p =2$ we need to compute in detail the maps in the Mayer-Vietoris long exact sequence.
As seen in section \ref{cohom246}, the element $x^{2i}y^{2i}$ generate a module 
$\Z_4 \subset H^{2j}(\Z_4;M_{8i})$.
Moreover the following relation holds in $H^{2j}(\Z_6;M_{8i})$ for the powers of the 
polynomial $x^2 + xy + y^2 = p_2$:
\begin{equation} \label{e:power}
p_2^{2i} = (2k+1) x^{2i}y^{2i} + \sum_{a \neq 2i} c_a (x^ay^{4i-a} + (-1)^a x^{4i-a}y^a)
\end{equation}
for some integer coefficients $k$ and $c_a$ for $a = 0, \ldots, 4i$. Let us call $\rho_{2i}$
the right term of the equation. The summands in the sommatory of $\rho_{2i}$ generate $\Z_2$ submodules of $H^{2j}(\Z_4;M_{8i})$
, hence $\rho_{2i}$ represents a class in $H^{2j}(\Z_4;M_{8i})$
that generates a submodule isomorphic to $\Z_4$. It follows that $\Z_4$ is in
the kernel of the map
$$
H^{2j}(\Z_4;M_{8i}) \oplus H^{2j}(\Z_6;M_{8i}) \to H^{2j}(\Z_2;M_{8i}).
$$
These are the only $\Z_4$ subgrups of $H^j(\SL; M_{i})$. In fact we have:
\begin{lem} \label{l:4tors_sl2Z}
Let $j\geq 2$. For $j$ odd or $i \neq 0 \mod 4$
the group $H^j(\SL; M_{2i})$ has no $4$-torsion while for $j$ even the group $H^j(\SL; M_{8i})$ 
has a subgroup isomorphic to $\Z_4$ that is uniquely determined modulo $2$-torsion elements.  
\end{lem}
\begin{proof}
From propositions \ref{z2cohom}, \ref{prop:Z_4}, \ref{prop:tabz4} and \ref{poincarez6} we have that the module $H^{j}(\Z_2;M_{2i})$ is trivial for $i$ odd and $j$ even and for $i$ even and $j$ odd, while $H^{j+1}(\Z_4;M_{2i})\oplus H^{j+1}(\Z_6;M_{2i})$ is trivial for $i$ odd and $j$ odd and for $i$ even and $j$ even. Hence from the Mayer-Vietoris long exact sequence we can see that for $j \geq 2$ the $\Z_4$ submodules  previously described are the only possible $\Z_4$ submodules of $H^{j}(SL_2(\Z);M_{i})_{(2)}$. 
\end{proof}

To determine completely the $2$-torsion cohomology $H^{j}(SL_2(\Z);M)_{(2)}$ we need to compute the rank of the kernel and of the image of the maps 
$$
H^1(\Z_4; M \otimes \Z_2) \oplus  H^{1}(\Z_6;M \otimes \Z_2) 
\stackrel{p_1}{\to} H^{1}(\Z_2;M \otimes \Z_2).
$$
and
$$
H^2(\Z_4; M \otimes \Z_2) \oplus  H^{2}(\Z_6;M \otimes \Z_2) 
\stackrel{p_2}{\to} H^2(\Z_2;M \otimes \Z_2).
$$
We recall that we have the following commuting diagrams for $k > 0$:
$$
\xymatrix@C=2pc@R=2pc{
H^k(\Z_4; M \otimes \Z_2) \oplus  H^{k}(\Z_6;M \otimes \Z_2) \ar[r]^(.63){p_k} \ar@{=}[d] & H^{k}(\Z_2;M \otimes \Z_2) \ar@{=}[d]\\
H^{2j+k}(\Z_4; M \otimes \Z_2) \oplus  H^{2j+k}(\Z_6;M \otimes \Z_2) \ar[r]^(.63){p_{2j+k}}  & H^{2j+k}(\Z_2;M \otimes \Z_2). 
}
$$
Hence the informations about $\ker p_1$ and $\im p_1$, together with the previous remarks  will allow us to determine completely the $2$-torsion cohomology $H^{j}(SL_2(\Z);M)_{(2)}$ by means of the Universal Coefficient Theorem.

%
%
It is convenient to start with the case $p_2$. We have that $$H^2(\Z_6; M \otimes \Z_2) = \Z_2[s_2, t_3, u_3]/(u_3^2 = s_2^3 + t_3^2 + t_3 u_3 )$$ where $s_2 = x^2 + xy + y^2$, $t_3 =xy(x+y)$ and $u_3 = x^3 + x^2y + y^3$ (see equation (\ref{prop:62}));
$$H^2(\Z_4; M \otimes \Z_2) = \Z_2[\sigma_1, \sigma_2]$$
where $\sigma_1 = x+y$ and $\sigma_2 = xy$ (see equation (\ref{prop:42}));
$$H^2(\Z_2; M \otimes \Z_2) = \Z_2[x,y]$$ (see proposition \ref{z2invariants}).

The map $p_2$ is the natural inclusion of submodules of $M$ and hence the image is the
submodule $\Z_2[\sigma_1, \sigma_2] \oplus u_3 \Z_2[\sigma_1, \sigma_2]$. The kernel is
isomorphic to the intersection of $\Z_2[s_2, t_3, u_3]$ and $\Z_2[\sigma_1, \sigma_2]$. Hence
$\ker p_2 = \Z_2[s_2, t_3]$. This proves:
\begin{lem}
The Poincar\'e series for $\ker p_2$ is 
$$
\frac{1}{(1-t^4)(1-t^6)}
$$
while the Poincar\'e series for $\im p_2$ is
\begin{equation*}
\frac{1+t^6}{(1-t^4)(1-t^6)}.
\end{equation*}

\vspace{-\baselineskip}
\qed
\end{lem}

In the case $p_1$ we have $$H^1(\Z_6; M \otimes \Z_2) = \Z_2[s_2, t_3, u_3]/(u_3^2 = s_2^3 + t_3^2 + t_3 u_3 )$$ where $s_2 = x^2 + xy + y^2$, $t_3 =xy(x+y)$ and $u_3 = x^3 + x^2y + y^3,$ 

$$H^1(\Z_4; M \otimes \Z_2) = \Z_2[x, xy]$$
and
$$H^1(\Z_2; M \otimes \Z_2) = \Z_2[x,y].$$ 
The equalities above are a consequence of the  fact that we can use a $2$-periodic resolution for $\Z_q$ and since the cohomology is torsion-free in degree greater that $0$ we have that the cohomology with coefficients in a vector space is $1$-periodic in degree greater than $0$. In particular $$H^1(\Z_q; M \otimes \Z_2) = H^2(\Z_q; M \otimes \Z_2) \; \; \; \; q =  2,4,6.$$
The generators that we provide are genuine representatives with respect to the standard
periodic resolution of $\Z_q$. This is trivial in the case of $\Z_2$, where the bondary maps of the complex are the null maps for our system of coefficients. For $\Z_4$ it is clear that we are providing a set of representatives for the quotient 
$$
\Z_2[x,y]/(T_4 - \Id)\Z_2[x,y]
$$
where $T_4$ acts on $\Z_2[x,y]$ by $$T_4:\left\{\begin{array}{lr} x \mapsto & y \\ y \mapsto & x \end{array} \right. .$$
For $\Z_6$ we show that $H^1(\Z_6; M \otimes \Z_2)$ and $H^2(\Z_6; M \otimes \Z_2)$ have the same generators. In fact, let $$T_6:\left\{\begin{array}{lr} x \mapsto & y \\ y \mapsto & x + y \end{array} \right. $$ be the homomorphism corresponding to the action of a generator of $\Z_6$ on $\Z_2[x,y]$, it is clear that the $T_6$-invariant elements in $\Z_2[x,y]$ are in
the kernel of the map $d^1 =\Id + T_6 + \cdots + T_6^5 = (1 + T_6 + T_6^2)(1+T)$. Moreover if $\psi\in \Z_2[x,y]$ is $T_6$-invariant and is in the image of $T_6-\Id$ we have that 
$$
3 \psi = (\Id + T_6 + T_6^2) \psi = (\Id + T_6 + T_6^2)(T_6 -\Id) \psi' = 
(T_6^3 - \Id) \psi' =0
$$
and so the intersection $\im\{T_6: \Z_2[x,y] \to \Z_2[x,y]\} \cap \Z_2[x,y]^{T_6}$ is trivial. An argument of dimension implies that $H^1(\Z_6; M \otimes \Z_2) = \Z_2[s_2, t_3, u_3]/(u_3^2 = s_2^3 + t_3^2 + t_3 u_3 )$.

In order to compute explicitly the map $p_1$ we need to compute the lifting of the inclusion homorphisms $Z_2 \into \Z_4$ and $\Z_2 \into \Z_6$ to the corresponding resolutions. Let $T_2 = T_4^2 = T_6^3$ the automorphism of $\Z_2[x,y]$ given by the generator of $\Z_2$, the have the following commuting diagrams for $\Z_4$:
$$
\xymatrix@C=2pc@R=2pc{
\Z_2[x,y] \ar[d]^{T_4 - \Id} \ar@{=}[rr]& & \Z_2[x,y] \ar[d]^{T_2-\Id} \\
\Z_2[x,y] \ar[rr]^{T_4+\Id}&& \Z_2[x,y] 
%
}
$$
and for $\Z_6$:
$$
\xymatrix@C=2pc@R=2pc{
\Z_2[x,y] \ar[d]^{T_6 - \Id} \ar@{=}[rr]&& \Z_2[x,y] \ar[d]^{T_2-\Id} \\
\Z_2[x,y] \ar[rr]^{T_6^2 + T_6 +\Id}&& \Z_2[x,y] .
%
}
$$
The map $i_4= T_4+\Id: \Z_2[x,y]^{T_4}$ acts as follows:
$$x^ay^y \mapsto x^ay^b+x^by^a \; \; \;\;\;\; \mbox{ for } a> b$$
and the image is $\im i_4 = s_1\Z_2[s_1,s_2]$.

The map $i_6 = T_6^2 + T_6 +\Id$ is the natural inclusion, as the map $T_6$ acts as the identity on the invariant elements and we are working modulo $2$.
The intersection $\im i_4 \cap \im i_6$ is the submodule 
$$s_1s_2\Z_2[s_1^2 + s_2, s_1s_2] \oplus \Z_2[s_2].$$
\begin{lem}
The Poincar\'e series for the kernel of $p_1$ is
\begin{equation*}
\frac{1}{1-t^4} + \frac{t^6}{(1-t^4)(1-t^6)} = \frac{1}{(1-t^4)(1-t^6)}.
\end{equation*}

\vspace{-\baselineskip}
\qed
\end{lem}
The computation of the Poincar\'e series $P^j_{SL_2(\Z),2}(t)$ is now strightforward.
We resume our result in the next proposition. We write $T_{SL_2(\Z),p}^{j}(t)$ for the Poincar\'e series for $H^{j}(SL_2(\Z);M_{i})_{(p)} \otimes \Z_p$.

\begin{prop} \label{p:2series}
Let us assume $j > 0$. The module $H^{j}(SL_2(\Z);M_i)_{(2)}$ is given as follows: 
\begin{itemize}
\item[a)] for $j$ even and $i = 0 \mod 8$ $H^{j}(SL_2(\Z);M_{i})_{(2)}$ is the direct sum of one module isomorphic to $\Z_4$ and some modules isomorphic to $\Z_2$; the generator of the module $\Z_4$ (up to elements of $2$-torsion) maps to the couple $$(p_{2i}^2, \rho_{2i}) \in 
H^{i+1}(\Z_4;M_{j})\oplus H^{i+1}(\Z_6;M_{j})$$
in the Mayer-Vietoris long exact sequence;
\item[b)] for all others $i,j$ $H^{j}(SL_2(\Z);M_{i})_{(2)}$ is the direct sum of a certain number of copies of $\Z_2$.
\end{itemize}
The series $T_{SL_2(\Z),2}^{j}(t)$ are given by:
\begin{eqnarray*}
T_{SL_2(\Z),2}^{1}(t) & = & \frac{t^4 + t^6 - t^{10} }{(1-t^4)(1-t^6)} \\
T_{SL_2(\Z),2}^{2i}(t) & = & \frac{1 - t^4 + 2 t^6 - t^8 + t^{12}}{(1-t^2)(1+t^6)(1-t^8)} \; \; \; \; \; \; \; \; i > 0 \\
T_{SL_2(\Z),2}^{2i+1}(t) & = & \frac{t^4(2-t^2+t^4 + t^6 - t^8)}{(1-t^2)(1+t^6)(1-t^8)} \; \; \; \; \; \; \; \; i > 0. 
\end{eqnarray*}

\vspace{-\baselineskip}
\qed
\end{prop}

The computations 
follow from the results of section \ref{cohom246}.
The modules $H^i(SL_2(\Z);M \otimes \Z_p))$ can be completely determined from the previous observations using the Universal Coefficients Theorem. 
In particular we have that the Poincar\'e series for $H^i(SL_2(\Z); M \otimes \Z_p)$ with 
respect to the grading of $M$ is given by the formula
$$
P_{SL_2(\Z),p}^i(t) = \sum_{j=0}^\infty \dim_{\F_p}H^i(SL_2(\Z); M \otimes \Z_p) t^j = T_{SL_2(\Z),p}^i(t) + T_{SL_2(\Z),p}^{i+1}(t) + T_{SL_2(\Z),0}^i(t).
$$

%% file: cohomB3.tex
\section{Cohomology of $B_3$} \label{s:cohomB3}
In order to compute the cohomology of $B_3$ we use the Serre spectral sequence already described in Section \ref{general} for the extension
$$
0 \to \Z \to B_3 \to SL_2(\Z) \to 1.
$$
We recall (equation (\ref{tworows})) that the $E_2$-page is represented in the following diagram:

\begin{center}
\begin{tabular}{l}
\xymatrix @R=1pc @C=1pc {
\H1 H^0(\SL,M) \ar[rrd]^(.7){d_2} & H^1(\SL,M) \ar[rrd]^(.7){d_2} & H^2(\SL,M) &  \cdots \\
H^0(\SL,M) & H^1(\SL,M) & H^2(\SL,M) &  \cdots
} \\
\end{tabular}
\end{center}
where all the other rows are zero. Recall that $H^i(\SL,M)$ is $2$-periodic and with no free 
part for $i \geq 2$.

The free part of $H^*(B_3; M)$ can easily be determined since the $E_2^{s,t}$ is torsion
for $s > 1$, hence the spectral sequence for $H^*(B_3; M \otimes \Q)$ collapses at $E_2$ and we have:
\begin{enumerate}
\item $H^0(B_3;M)\ =\  H^0(\SL,M)\ =\ \Z$ concentrated in degree $n=0$;
\item $H^1(B_3;M_0)\ =\  H^0(\SL,M_0)\ =\ \Z$;
\item $FH^1(B_3;M_n) \ =\ FH^1(\SL;M_n)  \ =\  \Z^{f_n}$ for $n > 0$;
\item $FH^2(B_3;M_n) \ =\ FH^1(\SL;M_n)  \ =\  \Z^{f_n}$.
\end{enumerate}
This proves what we stated in Theorem \ref{t:H^1} and \ref{t:H^2} for the free part.

Since the cohomology group $H^i(\SL,M)$ has only $2$ and $3$-torsion for $i \geq 2$, the same argument implies that for the $p$-torsion with $p>3$ the spectral sequence collapses at $E_2$. Hence for $p > 3$ we have the isomorphism 
$$
H^1(B_3;M)_{(p)} = H^2(B_3;M)_{(p)} = H^1(\SL;M)_{(p)} = \Ga^+_p[\first_p, \second_p]
$$
where $\deg \first_p = 2(p+1)$ and $\deg \second_p = 2p(p-1)$ as in section \ref{results}.

Moreover, since the first column of the spectral sequence is concentrated in degree $0$, we have that 
$$
H^1(B_3;M) = H^1(\SL;M) \oplus \Z
$$
and the contribute of the summand $\Z$ is concentrated in degree $0$.

For a complete description of the cohomology of $H^2(B_3;M)$ we have to study the behaviour of the spectral sequence for the primes $p=2$ and $p=3$.

We consider the subgroups $\Z_2$, $\Z_4$ and $\Z_6$ of $\SL$ from the presentation given in 
proposition \ref{presentation} and the induced extensions.  One has diagrams

\begin{equation}\label{e:3diags}
\xymatrix @R=1pc @C=1pc {
\Z\ar@{=}[r] \ar[d]^2 & \Z \ar[d] & & \Z\ar@{=}[r] \ar[d]^4 & \Z \ar[d] & & \Z\ar@{=}[r] \ar[d]^6 & \Z \ar[d]\\
\Z \ar[d] \ar@{^{(}->}[r] & B_3\ar[d] & &\Z \ar[d] \ar@{^{(}->}[r] & B_3\ar[d] & & \Z \ar[d] \ar@{^{(}->}[r] & B_3\ar[d] \\
\Z_2 \ar@{^{(}->}[r] & \SL & &\Z_4 \ar@{^{(}->}[r] & \SL & & \Z_6 \ar@{^{(}->}[r] & \SL \\
} 
\end{equation}
\noindent and the corresponding maps of spectral sequences give the commuting diagram

\centerline{
\xymatrix @R=1pc @C=1pc {
H^i(\SL;M) \ar[r]^{d_2} \ar[d]^{\Phi^i_m}& H^{i+2}(\SL;M) \ar[d]^{\Phi^{i+2}_m}\\
H^i(\Z_m; M) \ar[r]^{d_2} & H^{i+2}(\Z_m; M)\\
} 
}

\noindent for $m=2,4,6$.
%
We recall that the spectral sequences for $\Z \stackrel{m}{\to} \Z \to \Z_m$
collapses at $E_2$ and the differential
$$
H^i(\Z_m; M) \stackrel{d_2}{\to} H^{i+2}(\Z_m; M)
$$
is surjective for $i =0$ and isomorphisms for $i>0$.

For odd $n$ the 
$d_2$ differential in the spectral sequence for the cohomology of $B_3$ has a very simple behaviour:

\begin{lem} \label{l:odddiff}
For odd $n$ the differential:
$$
d_2: H^1(\SL;M_{2n}) \to H^3(\SL;M_{2n})
$$
is an isomorphism.
\end{lem}
\begin{proof}
For odd $n$ the groups $H^{2i}(\Z_2;M_{2n})$ are trivial and 
we have the following commuting diagram:

\centerline{
\xymatrix @R=1pc @C=0.6pc {
0 \ar[r]& H^{2i+1}(\SL;M_{2n}) \ar[d]^{d_2} \ar[r]  & H^{2i+1}(\Z_4;M_{2n})
\oplus
H^{2i+1}(\Z_6;M_{2n}) \eq[d]^{d_2}\ar[r] & H^{2i+1}(\Z_2;M_{2n}) \eq[d]^{d_2} \\
0 \ar[r]& H^{2i+3}(\SL;M_{2n})  \ar[r]  & H^{2i+3}(\Z_4;M_{2n})
\oplus
H^{2i+3}(\Z_6;M_{2n}) \ar[r] & H^{2i+3}(\Z_2;M_{2n}) \\
}
}
 
The statement follows from the Five Lemma.
\end{proof}

For a detailed computation of the differential $d_2$ for the degree $4n$ part of the spectral 
sequence we have to consider the representations of $\SL$ and $B_3$ over the module 
$M_n \otimes R.$  Here $R$ can be either $\Z_{(p)},$ the ring of integers localized at the prime $p,$ 
or the cyclic group $\Z_{p^i}$. We need to consider only $p=2$ or $p=3$.

\begin{subsection}{$2$-torsion}
Let us first focus on the $2$-torsion. 
We need to compare the Serre spectral sequence for $\Z \stackrel{2}{\to} \Z \to \Z_2$ 
and the spectral sequence for $B_3$ through the connecting homomorphism of 
the Mayer-Vietoris sequence. We do this by using the following result. 
\begin{prop}\label{p:ssmayerviet}
Let $G$ be any group of coefficients and let $$F \into E \stackrel{p}{\to} X$$ be a Serre fibration; suppose that the base space $X$ is the union of two open set $X = A \cup B$ and call $Y = A \cap B$ their intersection. 

Let $\mathrm{E}_2^{i,j}= H^i(X; H^j(F; G))$ (resp. ${\mathrm{E}(Y)}_2^{i,j}= H^i(Y; H^j(F; G))$) be the $\mathrm{E_2}$-term of the Serre spectral sequence for $F \into E \stackrel{p}{\to} X$ (resp. $F \into p^{-1}(Y) \stackrel{p}{\to} Y$).

The connecting homomorphism $$\delta^*: H^*(Y; H^j(F; G)) \to H^{*+1}(X; H^j(F; G))$$ in the Mayer-Vietoris long exact sequence associated to $X = A \cup B$ commutes with the spectral sequences differentials, hence it induces a map of spectral sequences $\delta^*: {\mathrm{E}(Y)} \to \mathrm{E}$. 
\end{prop}

We provide a proof of proposition \ref{p:ssmayerviet} in the Appendix \ref{appendix}.

Applying proposition \ref{p:ssmayerviet} to the decomposition $\SL = \Z_4 *_{\Z_2} \Z_6$ and to the spectral sequences associated to the extension $\Z \to B_3 \to \SL$ we
get for any $\SL$-module $N$ the following diagram:
\begin{equation}\label{e:doublelong}
\begin{split}
\xymatrix @R=1pc @C=0.6pc {
\cdots 
\ar[r]& 
H^{i-1}(\Z_2;N) \ar[d]^{d_2} \ar[r] & H^i(\SL;N) \ar[d]^{d_2} \ar[r] & H^i(\Z_4;N)
\oplus
H^i(\Z_6;N) \ar[d]^{d_2}\ar[r] & 
\cdots 
\\
\cdots 
\ar[r]& 
H^{i+1}(\Z_2;N)  \ar[r] & H^{i+2}(\SL;N) \ar[r] & H^{i+2}(\Z_4;N) 
\oplus
H^{i+2}(\Z_6;N) \ar[r] & 
\cdots 
\\
}
\end{split}
\end{equation}

Recall from Section \ref{cohom246} that for $n \equiv 0 \mod 4$ we have 
$$H^0(\Z_2;M_n) = M_n, H^1(\Z_2;M_n) = 0 \mbox{ and } H^2(\Z_2;M_n) = M_n\otimes \Z_2.$$ 
Let now $N= M_n \otimes \Z_{(2)}$. As $H^i(\Z_6;N)$ is
trivial for $i$ odd (see proposition \ref{poincarez6})
the diagram (\ref{e:doublelong})
specializes to:
\begin{equation} \label{e:bigdiagram}
\begin{split}
\xymatrix @R=1pc @C=1pc {
0  \ar[r] & N^{\SL} \ar[d]^{d_2} \ar[r] & N^{\Z_4}
\oplus
N^{\Z_6} \ar[d]^{d_2}\ar[r] & 
\\
0 \ar[r]& H^{2}(\SL;N)  \ar[r] & H^2(\Z_4;N) 
\oplus
H^{2}(\Z_6;N) \ar[r] &
\\
}
\\
\xymatrix @R=1pc @C=1pc {
\ar[r] & H^0(\Z_2; N) \ar[d]^{d_2} \ar[r] &   H^1(\SL;N) \ar[r] \ar[d]^{d_2} &
H^1(\Z_4;N) \ar[r] \ar[d]^{d_2} &
0 \\
\ar[r] & H^2(\Z_2; N) \ar[r] & H^3(\SL;N) \ar[r] &
H^3(\Z_4;N) \ar[r] &
0 \\
}
\end{split}
\end{equation}

\begin{lem} \label{l:neven}
For even $n$ and $N = M_{2n} \otimes \Z_{(2)}$ we have the following equalities:
\begin{equation} \label{e:neven}
\begin{split}
& \ker (d_2:H^1(\SL;N)  \to  H^3(\SL;N)) = \\
& = 2 \im (H^0(\Z_2;N)  \to  H^1(\SL;N)) = \\ 
& = 2 \ker(H^1(\SL; N)  \to  H^1(\Z_4; N)).
\end{split}
\end{equation}
Moreover $H^3(\SL;N) = H^1(\SL;N) \otimes \Z_2$ and the differential 
$$d_2:  H^1(\SL;N) \to H^3(\SL;N)$$ is the projection map $H^1(\SL;N) \to H^1(\SL;N) \otimes \Z_2$.
\end{lem}
\begin{proof}
For even $n$ and odd $i$ we have $H^i(\Z_4; M_{2n}) = 0 $ 
if $n \equiv 0 \mod 4$ and 
$H^i(\Z_4; M_{2n}) = \Z_2 $ if $n \equiv 2 \mod 4$. Moreover the differential 
$$d_2:H^1(\Z_4; M_{2n}) \to H^3(\Z_4; M_{2n})$$
is an isomorphism. The inclusion map $i_4: \Z_4 \to \Z_4 *_{\Z_2} \Z_6 = \SL$ has a right 
inverse $\nu_4:\SL \to \Z_4$ defined on $\Z_4$ as the identity and on $\Z_6$ as the multiplication by $2$.  
Hence the exact sequences of diagram (\ref{e:bigdiagram}) split, respectively in $H^1(\SL; N)$ and $H^3(\SL; N)$.
Finally it is easy to verify that the differential 
$$
d_2: H^0(\Z_2;M_{2n}) \to H^2(\Z_2;M_{2n}) 
$$
is the quotient mod-$2$ map $M_{2n} \to M_{2n} \otimes \Z_2$. 
The lemma follows from diagram (\ref{e:bigdiagram}).
\end{proof}

We write here $\beta_{p^r}$ as the $r^{\mbox{th}}$-order Bockstein operator associated to the sequence of coefficients
$$
0 \to \Z_{p^r} \stackrel{\times p^r}{\to} \Z_{p^{2r}} \to \Z_{p^r} \to 0.
$$

We 
can decompose $H^1(\SL;M_n \otimes \Z_{(2)})$, for $n \equiv 0 \mod 4$
, as a direct sum of a free $\Z_{(2)}$-module of rank $f_n$ (see Theorem 
\ref{t:H^1}) and some
modules of the form $\Z_{2^k}$. We can choose a decomposition
\begin{equation} \label{e:decomp} 
H^1(\SL;M_n\otimes \Z_{(2)}) = \Z_{(2)}^{f_n} \oplus \bigoplus_k \Z_{2^k}^{d_k} 
\end{equation}
and an analogous decomposition for $H^1(\Z_4;M_n\otimes \Z_{(2)})$ 
such that the map $$H^1(\SL;M_n\otimes \Z_{(2)}) \to H^1(\Z_4;M_n\otimes \Z_{(2)})$$ 
is diagonal with respect to the two decompositions.
According to lemma \ref{l:neven} the decomposition of $H^1(\SL;M_n\otimes \Z_{(2)})$ induces a
decomposition of $H^3(\SL;M_n\otimes \Z_{(2)}) = H^1(\SL;M_n\otimes \Z_{(2)}) \otimes \Z_2$.

From the previous lemmas the differential $d_2$ is completely determined and hence also 
the term $E_\infty$, as the spectral sequence lives only on rows $0$ and $1$ and then collapses at $E_3$.

\begin{lem} \label{l:EinftyZ2}
In the Serre spectral sequence for $H^*(B_3;M_n\otimes \Z_{(2)})$ for $n 0 \equiv \mod 4$ the $E^\infty$-term is given by:
\begin{center}
\begin{tabular}{l}
\xymatrix @R=1pc @C=1pc {
\g1 \Z_{(2)} & 2 H^1(\SL;M_n \otimes \Z_{(2)}) & 0 \\
\Z_{(2)} & H^1(\SL;M_n \otimes \Z_{(2)}) & H^2(\SL;M_n \otimes \Z_{(2)})\\
} \\
\end{tabular} 
\end{center}\qed
\end{lem}
Now we consider the extension problem associated to the $E_\infty$-term, in order to determine the cohomology groups 
$H^*(B_3;M\otimes \Z_{(2)})$. 

\begin{thm} \label{t:ss_2tors}
We consider the extension problem for the $E^\infty$-term of the spectral sequence of 
Lemma \ref{l:EinftyZ2} above. 
We can choose a decomposition in the form of Equation (\ref{e:decomp}) such that for the
induced decomposition of $2 H^1(\SL;M_n \otimes \Z_{(2)})$ we have the following cases:
\begin{itemize}
\item[a)] for $n \equiv 0 \mod 8$ 
there is a summand $\Z_{2^{i}} \subset 2 H^1(\SL;M_n \otimes \Z_{(2)})$ corresponding, via Universal Coefficient Theorem, 
to the generator $\second_2^{2^ih} \in H^0(\SL;M \otimes \Z_2)$ where $n = 2^{i+1}h$ 
($2 \nmid h$, $i>0$); 
the summand has a lifting to the module $\Z_{2^{i+2}}$ through the extension with the module $\Z_4$ corresponding 
to the generator $$(x^2y^2)^{2^{i-1}} = e_4^{2^{i-1}} \in H^2(\Z_4;M \otimes \Z_{(2)});$$
\item[b)] for all others torsion summands in $2 H^1(\SL;M_n \otimes \Z_{(2)})$, that we can suppose to be in the form $\Z_{2^{i-1}}$,
there is a lifting to the module $\Z_{2^{i}}$ through the extension with a submodule $\Z_2$ corresponding to a generator in $H^2(\SL;M\otimes \Z_{(2)} )$;
\item[c)] for all the free summands $\Z_{(2)} \subset 2 H^1(\SL;M_n \otimes \Z_{(2)})$ there is no lifting and the extension is trivial.
\end{itemize}
\end{thm}
\begin{proof}
We study the spectral sequence $\overline{E}_2$ 
$$
\overline{E}_2^{i,j} = H^i(\SL; M_n\otimes \Z_{2^k}) \mbox{ for } j=0,1
$$
and zero otherwise, that converges to $H^*(B_3;M_n\otimes \Z_{2^k})$ for any integer $k>0$. 
The dependence on the integer $k$ is understood in the notation.
As before, the spectral sequence collapses at $\overline{E}_3$.
We compute the cardinality of $ \overline{E}_3^{1,1} \oplus \overline{E}_3^{2,0}$ 
that is the cardinality of $H^2(B_3; M_n\otimes \Z_{2^k})$.

We will need to consider also the $\overline{E}(\Z_m)_2$ 
spectral sequences for $H^*(\Z; M_n\otimes \Z_{2^k})$ obtained 
by the extensions given
in the diagrams of equation (\ref{e:3diags}), induced by the inclusion $\Z_m \into \SL$, $m= 2,4,6$. 

The Mayer-Vietoris exact sequence gives maps of spectral 
sequences
$$
\Phi_4: \overline{E}_2 \to \overline{E}(\Z_4)_2,
$$
$$
\Phi_6: \overline{E}_2 \to \overline{E}(\Z_6)_2,
$$
$$
\Delta_2: \overline{E}(\Z_2)_2 \to \overline{E}_2
$$
where the last map is the homomorphism of proposition \ref{p:ssmayerviet} and has a degree shift coherently with the Mayer-Vietoris exact sequence connecting homomorphism.

Moreover we need to consider the map
$$
\pi_*: E_2 \to \overline{E}_2
$$
induced by the morphism of coefficients $\Z \stackrel{\pi}{\to} \Z_{2^k}$ and the Bockstein homomorphism
$$
\beta_{2^k}: \overline{E}_2 \to \overline{E}_2
$$
and, for $m = 2,4,6$,
$$
\beta_{2^k}: \overline{E}_2(\Z_{m}) \to \overline{E}_2(\Z_{m})
$$
defined before.

Here we use $a \vee b$ for $\max(a,b)$ and $a \wedge b$ for $\min(a,b)$.

Let $\omega = \first_2^a \second_2^b$ any monomial in $\first_2, \second_2$. Let $2^i$ be the greatest power of $2$ that divides both $a$ and $b$. The polynomial $\omega$ is $\SL$-invariant modulo $2^{i+1}$, hence it generates a submodule $\Z_{2^{i+1}} \subset H^0(\SL;M \otimes \Z_{2^{i+1}})$. More generally, for $k \leq i$ the polynomial $\omega$ generates
a submodule $\Z_{2^{k}} \subset H^0(\SL;M \otimes \Z_{2^{k}})$, while for $k \geq i+1$ the
polynomial $2^{k-i-1}\omega$ is $\SL$-invariant modulo $2^k$ and generates
a submodule $\Z_{2^{i+1}} \subset H^0(\SL;M \otimes \Z_{2^{k}})$. We remark that all these submodules are also direct summands.

Define $\delta = \max(k-i-1, 0)$. 

The Bockstein homomorphism $$\beta_{2^k}: H^0(\SL;M \otimes \Z_{2^{k}}) \to H^1(\SL;M \otimes \Z_{2^{k}})$$ maps the polynomial $2^\delta \omega$ as follows.
For $k \geq i+1$ the image $\beta_{2^k}(2^\delta \omega)$ generates a $\Z_{2^{i+1}}$
that is a direct summand of $H^1(\SL;M \otimes \Z_{2^{k}})$, while for $k \leq i$
the image $\beta_{2^k}(2^\delta \omega)$ is $2^{i+1-k}$-times the generator of a module
$\Z_{2^{k}}$ that is a direct summand of $H^1(\SL;M \otimes \Z_{2^{k}})$. In particular for
$i+1 \geq 2k$, $\beta_{2^k}(2^\delta \omega) =0$.

In both cases ($k \geq i+1$ or $k \leq i$), the module 
$\Z_{2^{k \wedge (i+1)}} \subset H^1(\SL;M \otimes \Z_{2^{k}})$ considered above is 
the image through $\pi_*$ of a module $\Z_{2^{i+1}}$ that is a direct summand of $H^1(\SL;M \otimes \Z_{(2)})$; we call $\widetilde{ \omega} \in H^1(\SL;M \otimes \Z_{(2)})$ the generator of this module and we have
$$
2^{i+1-k-\delta} \pi_* \widetilde{ \omega} = \beta_{2^k}(2^\delta \omega).
$$

Let us now focus on the polynomial $\omega = \second_2^{2^ih}$ for odd $h$ and $i > 0$. 
We need to prove the following:
\begin{lem} \label{l:d2a}
The differential $d_2: \overline{E}_2^{0,1} \to \overline{E}_2^{2,0}$
maps the module $\Z_{2^{k\wedge (i+1)}}$ generated by $2^\delta \second_2^{2^ih}$ surjectively to a module isomorphic to $\Z_2$ if $k \geq i+2$ or $k=1$ and to $\Z_4$ if $2 \leq k \leq i+1$.
\end{lem}
\begin{proof}
Let $k \leq i+1$ and set $\omega = \second_2^{2^ih}$. Then $\delta = 0$ and it is easy to verify that monomial $(xy)^{2^ij}$ appears with an odd coefficient in the algebraic expansion of $2^\delta \omega = \second_2^{2^ih}$. Hence the cocycle $\Phi_4(2^\delta \second_2^{2^ih})$ generate a direct summand $\Z_{2^k}$ in $\overline{E}(\Z_4)_2$ and for $k \neq 1$ the differential $d_2$ maps the cocycle $\Phi_4(2^\delta \second_2^{2^ih})$ to the generator of a module $\Z_{4}$ in $\overline{E}(\Z_4)_2$. Since $H^2(\SL;M \otimes \Z_{2^k})$ has no $8$-torsion, the claim for $2 \leq k \leq i+1$ follows from the commutativity of the maps $d_2$ and $\Phi_4$. Se the diagram in table \ref{tab:a}.
The same argument shows that for $k =1$ the cocycle $d_2 (2^\delta \second_2^{2^ih})$ generate a module $\Z_2\subset \overline{E}_2^{2,0}$.

\begin{table}[htbe]
\begin{center}
\begin{tabular}{l}
\xymatrix @R=1pc @C=1pc {
& & & \g2 \overline{E}(\Z_4)_2^{0,1} \ar@/^/[rrd]^{d_2} & & \\
\g1\overline{E}_2^{0,1} \ar[rrru]^{\Phi_4} \ar@/^/[rrd]^{d_2} & & & & & \overline{E}(\Z_4)_2^{2,0} \\
& &  \overline{E}_2^{2,0} \ar[rrru]_{\Phi_4} & & 
} \\
\end{tabular} 
\end{center}
\caption{}\label{tab:a}
\end{table}

Now suppose $k \geq i +2$. 
In this case $\delta > 0$ and by the same argument as before, we have that the class $d_2 \Phi_4(2^\delta \second_2^{2^ih})$ has no $4$-torsion. From the description of the cohomology of $\SL$ given in lemma  \ref{l:4tors_sl2Z} it follows that also the class $d_2 (2^\delta \second_2^{2^ih})$ has no $4$-torsion.
In fact the morphism $\Phi_4$ maps any class of $4$-torsion in $\overline{E}_2^{2,0}$ isomorphically to a class of $4$-torsion in $\overline{E}(\Z_4)_2^{2,0}$.

In order to complete the proof of the lemma we need to show that $d_2 (2^\delta \second_2^{2^ih})$ is non-zero for $k \geq i +2$. We consider the diagram of table \ref{tab:b}.
\begin{table}[htb]
\begin{center}
\begin{tabular}{l}
\xymatrix @R=1pc @C=1pc {
& & & & \g1 E_2^{1,1}\ar@/^/[rrd]^{d_2} \ar[dlll]_{\pi_*} & & \\
\G1 \overline{E}_2^{0,1}\ar@/^/[rrd]_{d_2} \ar[r]^{\beta_{2^k}} & \overline{E}_2^{1,1}\ar@/^/[rrd]^{d_2} & & & & &E_2^{3,0} \ar[dlll]^{\pi_*} \\
& & \overline{E}_2^{2,0} \ar[r]^(0.45){\beta_{2^k}}& \overline{E}_2^{3,0}& & &
} \\
\end{tabular} 
\end{center}
\caption{}\label{tab:b}
\end{table}
From the previous computations we have $\beta_{2^k}(2^\delta \omega) = \pi_* \widetilde{ \omega}$
and the Bockstein map $\beta_{2^k}$ maps isomorphically the submodule $\Z_{2^k}$ generated by $2^\delta \omega$ to the submodule $\Z_{2^k}$ generated by $\pi_* \widetilde{ \omega}$. Moreover from 
lemma \ref{l:neven} the differential $d_2$ maps $\pi_* \widetilde{ \omega}$ to the generator of a submodule $\Z_2$. Hence, by the commutativity of the differential $d_2$ and the Bockstein map $\beta_{2^k}$, the cohomology class $d_2(2^\delta \omega)$ is non-zero. The lemma follows.
\end{proof}

In case $\omega$ is different from an even power of $\second_2$, we have:
\begin{lem} \label{l:d2b}
Let $\omega = \first_2^a \second_2^b$ be a monomial in $\first_2$ and $\second_2$ different from an even power of $\second_2$. 
The differential $d_2: \overline{E}_2^{0,1} \to \overline{E}_2^{2,0}$
maps the module $\Z_{2^{k\wedge (i+1)}}$ generated by $2^\delta \omega$ surjectively to a module isomorphic to $\Z_2$.
\end{lem}
\begin{proof}
Let $k \geq i+1$. In this case we can use the Bockstein morphism $\beta_{2^k}$ and compare the spectral sequence $ \overline{E}_2$ with $E_2$ by means of the map $\pi_*$.
The Bockstein $\beta_{2^k}$ maps the submodule $\Z_{2^k}$ generated by $2^\delta \omega$ isomorphically to
the submodule $\Z_{2^k}$ generated by $\pi_* \widetilde{ \omega}$ and from 
lemma \ref{l:neven} the differential $d_2$ maps $\pi_* \widetilde{ \omega}$ to the generator of a submodule $\Z_2$. The argument in the proof of the previous lemma implies that the class $d_2 (2^\delta \omega)$
has no $4$-torsion.

Now we assume $k \leq i$ and hence $\delta = 0$. If $a$ is odd then $i=0$ and there is nothing to prove. If $a =0$ then $b$ is odd and again $i=0$ and there is nothing to prove. If $a \neq 0$ and $a$ even, a symmetry argument that exchanges $x$ and $y$ shows that 
the power of $xy$ that appears in the algebraic expansion of $\omega$
has even coefficient. Hence the class $d_2\Phi_4(\omega)$ has at most $2$-torsion and can be eventually zero. We need to prove that $d_2(\omega)$ is non-zero.
It is easy to see that the monomials $x^{2a+2b}y^a$ and $x^ay^{2a+2b}$ appear in the algebraic expansion of $\omega$ with coefficient $1$. As we are assuming that $a$ is even we have that $\Phi_4(\omega)$ has non-trivial projection to the $\Z_4$-invariant submodule of $\Z_{2^k}[x,y]$ generated by $x^{2a+2b}y^a$ and $x^ay^{2a+2b}$ and since $d_2(x^{2a+2b}y^a + x^ay^{2a+2b}) \neq 0$ we have also that $d_2 \Phi_4(\omega) \neq 0$.
\end{proof}

\begin{lem} \label{l:d2c}
The differential $d_2: \overline{E}_2^{0,1} \to \overline{E}_2^{2,0}$ induces an injective map $\overline{d_2}: \overline{E}_2^{0,1} \otimes \Z_2 \to \overline{E}_2^{2,0} \otimes \Z_2$.
\end{lem}

\begin{proof}
We define the map $\overline{\beta}_{2^k}: \overline{E}_2^{0,1}  \otimes \Z_2 \to \overline{E}_2^{1,1} \otimes \Z_2 $ induced by the Bockstein $\beta_{2^k}$ and we consider a splitting $\overline{E}_2^{0,1} = \ker \overline{\beta}_{2^k} \oplus \coker \overline{\beta}_{2^k}$. 

The kernel of $\overline{\beta}_{2^k}$ is generated by those classes $2^\delta \first_2^a \second_2^b$ such that $\delta =0$, that is $k \geq i+2$ where $2^i$ is the greatest power of $2$ that divides both $a$ and $b$. It follows that the restriction of $\Phi_4$ to $\ker \overline{\beta}_{2^k}$ is injective. Moreover it is easy to verify that the induced map
$$
\overline{d_2}: {\overline{E}(\Z_4)}_2^{0,1}  \otimes \Z_2 \to {\overline{E}(\Z_4)}_2^{1,1} \otimes \Z_2 
$$
is injective. Hence the restriction of $d_2$ to $\ker \overline{\beta}_{2^k} \subset \overline{E}_2^{0,1}$ is injective and the image is contained in the kernel of the map
$$
\overline{\beta}_{2^k}: \overline{E}_2^{2,0}  \otimes \Z_2 \to \overline{E}_2^{3,0} .
$$

The image of $\overline{\beta}_{2^k}: \overline{E}_2^{0,1}  \otimes \Z_2 \to \overline{E}_2^{1,1} \otimes \Z_2 $ is contained in the image of the projection 
$$
\overline{\pi}_*: E_2^{1,1} \otimes \Z_2 \to \overline{E}_2^{1,1}  \otimes \Z_2
$$
and since the map 
$$
\overline{d_2}: E_2^{1,1} \otimes \Z_2 \to E_2^{3,0} \otimes \Z_2
$$
is an isomorphism it follows that the restriction of the $d_2$ to $\coker \overline{\beta}_{2^k} \subset \overline{E}_2^{0,1}$ is injective and the image is contained
in the cokernel of $\overline{\beta}_{2^k}$.

The claim of the lemma follows.
\end{proof}

The previous arguments together with lemmas \ref{l:d2a} and \ref{l:d2b} determine completely the differential $d_2:\overline{E}_2^{0,1} \to \overline{E}_2^{2,0}$. 

Now we focus on $d_2:\overline{E}_2^{1,1} \to \overline{E}_2^{3,0}$. 

We consider the splitting of the module $\overline{E}_2^{1,1} = T_1 \oplus F_1 \oplus T_2$, where $T_1$ is the image, through $\pi_*$ of the torsion subgroup of $E_2^{1,1}$; $F_1$
is the image, through $\pi_*$ of the torsion-free part of $E_2^{1,1}$ according to the chosen decomposition in Equation \ref{e:decomp}; $T_2$ is a subgroup corresponding, through
the Bockstein homomorphism and the map $\pi_*$, to the torsion subgroup in $E_2^{2,1}$.

We recall from the previous results that the differential $d_2:E_2^{1,1} \to E_2^{3,0}$ 
is the quotient modulo $2$ for degree multiple of $4$ (see 
lemma \ref{l:neven}) and an isomorphism for degree $2n$, $n$ odd,
(see lemma \ref{l:odddiff}) with a $2$-torsion
module. Since the map 
$\pi_*$ induces a surjection $$E_2^{1,1} \to F^1 \oplus T^1 \subset \overline{E}_2^{1,1}$$ and the 
map $\pi_*: E_2^{3,0} \to \overline{E}_2^{3,0}$ is injective. We can also remark that $F^1 \oplus T^1$ in in the kernel of the Bockstein operator $\beta_{2^k}$. Hence we have the following
\begin{lem} \label{l:E11mod2}
The differential $d_2:\overline{E}_2^{1,1} \to \overline{E}_2^{3,0}$ acts on the submodule
$F^1 \oplus T^1 \subset \overline{E}_2^{1,1}$ as quotient modulo $2$:
$$
d_2:F^1 \oplus T^1 \to (F^1 \oplus T^1) \otimes \Z_2 \subset \overline{E}_2^{3,0}.
$$
The image $\beta_{2^k}(F^1 \oplus T^1)$ is contained in the kernel of the Bockstein operator $\beta_{2^k}$.
\qed
\end{lem}

Finally we consider the diagram of table \ref{tab:c} 
\begin{table}[htb]
\begin{center}
\begin{tabular}{l}
\xymatrix @R=1pc @C=1pc {
\G1 \overline{E}_2^{1,1}\ar@/^/[rrd]_{d_2} \ar[r]^{\beta_{2^k}} & \overline{E}_2^{2,1}\ar@/^/[rrd]^{d_2} & & \\
& & \overline{E}_2^{3,0} \ar[r]^(0.45){\beta_{2^k}}& \overline{E}_2^{4,0} 
} \\
\end{tabular} 
\end{center}
\caption{}\label{tab:c}
\end{table}
and we look at the submodule $T^2 \subset \overline{E}_2^{1,1}$. Note that $T^2$ is a $\Z_4$-module. In particular for $k \geq 2$ the restriction of the Bockstein map $\beta_{2^k}:\overline{E}_{2}^{1,1} \to \overline{E}_2^{2,1}$ to the submodule $T^2 \subset \overline{E}_{2}^{1,1}$
is injective. Moreover, since $B_3$ has homological dimension $2$, the differential  $d_2:\overline{E}_2^{2,1} \to \overline{E}_2^{4,0}$ is an
isomorphism.
It follows that for $k \geq 2$ the restriction of the differential $d_2:\overline{E}_2^{1,1} \to \overline{E}_2^{3,0}$ to the submodule $T^2 \subset \overline{E}_2^{1,1}$ is injective.
Since the Bockstein operator $\beta_{2^k}$ is injective in $T^2 \subset \overline{E}^{1,1}_2$  we have also that the image $d_2(T^2)$ projects injectively to the cokernel of the Bockstein and hence for $k \geq 2$.

For $k = 1$ we can consider the long exact sequence of equation (\ref{e:doublelong}) for $N = M_n \otimes \Z_2$. For $i \geq 0$ and $m = 2,4,6$ the differential
$$
d_2: H^i(\Z_m; M_n \otimes \Z_2) \to H^{i+2}(\Z_m; M_n \otimes \Z_2)
$$
is always an isomorphism. Hence, by the five lemma, it follows that the differential
$$
d_2: H^i(\SL; M_n \otimes \Z_2) \to H^{i+2}(\SL; M_n \otimes \Z_2)
$$
is an isomorphism too.

In general we have proved:

\begin{lem}
For any $k \geq 1$ the kernel of the map $d_2:\overline{E}_2^{1,1} \to \overline{E}_2^{3,0}$
is equal to the kernel of the restriction
$$
d_2:F^1 \oplus T^1 \to (F^1 \oplus T^1) \otimes \Z_2 \subset \overline{E}_2^{3,0}.
$$
\qed
\end{lem}

For the computation of the cardinality of $H^2(B_3; M_n\otimes \Z_{2^k})$ we define the function
$$ f(k):= \log_2(|H^2(B_3; M_n\otimes \Z_{2^k})|)
$$
and since $B_3$ has homological dimension $2$ it holds
$$
f(k)= \log_2(|H^2(B_3; M_n)\otimes \Z_{2^k}|)
$$
and the function $f(k)$ determines the group $H^2(B_3; M_n) \otimes \Z_{(2)}$. 
If we suppose that 
$$
H^2(B_3; M_n)\otimes \Z_{(2)} \simeq \Z_{(2)}^{\oplus f_n} \oplus \bigoplus_t \Z_{2^t}^{\oplus \mu_t}
$$
where $f_n$ is the integer defined in Theorem \ref{t:H^1} and the $\mu_t$ are non-negative 
integers, we have 
$$
f(k) = f_n k + \sum_t \mu_t (t \wedge k)
$$
and we need to determine the function $f(k)$ only up to constants in order to get 
the integers $\mu_t$.

The spectral sequence $\overline{E}_r$ collapses at the term $\overline{E}_3$ and we have
$
|H^2(B_3; M_n\otimes \Z_{2^k})| = |\overline{E}_3^{1,1} \oplus\overline{E}_3^{2,0}|
$ and 
\begin{equation} \label{e:log2}
f(k) = \log_2(|\overline{E}_3^{1,1}|) + \log_2(|\overline{E}_3^{2,0}|).
\end{equation}
The term $\log_2(|\overline{E}_3^{1,1}|)$ can be computed using Lemma \ref{l:E11mod2} 
and the description of $H^1(\SL; M_n)_{(p)}$ given in Theorem \ref{ptorsionH1}. We obtain:
$$
\log_2(|\overline{E}_3^{1,1}|) = k f_n + \sum_{
\substack{
{}_{a,b \in \N,}\\
{}_{6a+4b =n}
}
} ((\lambda_2(a,b)+1)\wedge k -1 )
$$
where $\lambda_2(a,b)$ is the greatest integer such that $2^{\lambda_2(a,b)} \mid \gcd(a,b)$.

From the description of the differential $d_2: \overline{E}_2^{0,1} \to \overline{E}_2^{2,0}$
given in lemmas \ref{l:d2a}, \ref{l:d2b} and \ref{l:d2c} we have that, up to constants, the cardinality of $\overline{E}_3^{2,0}$ follows by:
$$
\log_2(|\overline{E}_3^{1,1}|) = \sum_{2^{i+2}h = n} \chi_{\{k \geq i+2\}}.
$$
Hence, up to constants,  
$$
f(k) = k f_n + {\sum}_1 ((\lambda_2(a,b)+1)\wedge k -1 ) + {\sum}_2 (i+2)\wedge k -1 )
$$
where the sum $\sum_1$ is given for $a,b \in \N$ such that $6a+4b =n$ and $a \neq 0$ or $b$ odd and the sum $\sum_2$ is given for $i,h \in \N$ such that $i > 0$, $h$ is odd and $2^{i+1}h = n$.

From the description of the spectral sequence it follows that the extension involves the generator described in the statement of the theorem \ref{t:ss_2tors}.
\end{proof}

As a consequence of Theorem \ref{t:ss_2tors} we have the first part of Theorem \ref{t:B_3primes23} that is the result:
\begin{thm}\label{t:H2_2tors}
We have the following isomorphism:
$$H^2(B_3; M)_{(2)} = (\Ga^+_2[\first_2, \second_2] \oplus \Z[\overline{\second}_2^2])/\sim$$ 
Here $\overline{\second}_2$ is a new variable of degree $4$;  
the quotient module is defined by the relations $\frac{\second_2^n}{n!} \sim 2 \overline{\second}_2^n$ 
for $n$ even and $\frac{\second_2^n}{n!} \sim 0$ for $n$ odd. \qed
\end{thm}

\end{subsection}


\begin{subsection}{$3$-torsion}

Let now consider only the $3$-torsion of $H^2(B_3;M_{2n})$. 
For odd $n$ the $3$-torsion part of $H^i(\SL;M_{2n})$ is trivial, 
so we investigate the Serre spectral sequence and the differential 
$$
d_2: H^1(\SL;N) \to H^3(\SL;N)
$$
where now $N$ is the module $M_{2n} \otimes \Z_{(3)}$ for $n$ even.
In this case from diagram (\ref{e:doublelong}) we get
\begin{equation}\label{e:double3}
\begin{split}
\xymatrix @R=1pc @C=0.6pc {
\cdots \ar[r]& H^{0}(\Z_2;N) \ar[r] & H^1(\SL;N) \ar[d]^{d_2} \ar[r] & 
H^1(\Z_6;N) \ar[d]^{d_2}\ar[r] & 
0 \\
& 0  \ar[r] & H^{3}(\SL;N) \ar[r] & 
H^{3}(\Z_6;N) \ar[r] & 
0 \\
}
\end{split}
\end{equation}

Again the spectral sequence collapses at $E_3$ and we have:

\begin{lem} \label{l:EinftyZ3}
In the Serre spectral sequence for $H^*(B_3;M_n\otimes \Z_{(3)})$ the $E^\infty$-term is given by:
\begin{center}
\begin{tabular}{l}
\xymatrix @R=1pc @C=1pc {
\g1 \Z_{(3)} & K & 0 \\
\Z_{(3)} & H^1(\SL;M_n \otimes \Z_{(3)}) & H^2(\SL;M_n \otimes \Z_{(3)})\\
} \\
\end{tabular} 
\end{center}
where $K = \ker (H^1(\SL;M_n\otimes \Z_{(3)}) \to H^1(\Z_6; M_n\otimes \Z_{(3)})).$
\qed
\end{lem}

Similar to the case  $p=2,$ the final step is to consider the extension problem associated to the $E_\infty$-term.
We can assume also for $p=3$ to have a decomposition
\begin{equation} \label{e:decomp3} 
H^1(\SL;M_n\otimes \Z_{(3)}) = \Z_{(3)}^{f_n} \oplus \bigoplus_k \Z_{3^k}^{d_k} 
\end{equation}
and an induced decomposition on  $$H^1(\Z_6; M_n\otimes \Z_{(3)})) = H^{3}(\SL;M_n\otimes \Z_{(3)}) $$ such that 
the differential $d_2$ is diagonal with respect to the two decompositions.

As for the $2$-torsion, we use $\overline{E}_2$ for the spectral sequence defined by 
$$
\overline{E}_2^{i,j} = H^j(\SL; M_n\otimes \Z_{3^k}) \mbox{ for } i=0,1
$$
and zero otherwise, that converges to $H^*(B_3;M_n\otimes \Z_{3^k})$ and  also the $\overline{E}(\Z_m)_2$ 
spectral sequences for $H^*(\Z; M_n\otimes \Z_{2^k})$. 

The two spectral sequences are connected by the map
$$
\Phi_6: \overline{E}_2 \to \overline{E}(\Z_6)_2
$$
induced by the inclusion $\Z_6 \into \SL$ given
in the third diagram of equation (\ref{e:3diags}).

\begin{prop}
Let us consider the generator $\first_3\second_3^{3^ik} \in H^0(\SL;M \otimes \Z_3)= \overline{E}_2^{0,1}$. The corresponding module
$\Z_3 \subset H^1(\SL;M_n\otimes \Z_{(3)})= E_2^{1,1}$ with $n = 6\cdot 3^ik+4$ is mapped isomorphically by the differential $d_2$
to the module $\Z_3 = H^3(\SL;M_n\otimes \Z_{(3)}) = E_2^{3,0}$.
\end{prop}
\begin{proof}
The statement follows considering the spectral sequence $\overline{E}_2^{i,j}$ with coefficient in $M_n \otimes \Z_3$
and the spectral sequence $\overline{E}(\Z_6)_2^{i,j}$, both for $k=1$ and the maps of the diagram in table 
\ref{tab:d}. 

\begin{table}[htb]
\begin{center}
\begin{tabular}{l}
\xymatrix @R=1pc @C=1pc {
& & & & \g1 E_2^{1,1}\ar@/^/[rrd]^{d_2} \ar[dlll]_{\pi_*} & & \\
\G1 \overline{E}_2^{0,1} \ar[dd]^{\Phi_6} \ar@/^/[rrd]_(0.4){d_2} \ar[r]^{\beta_{3}} & \overline{E}_2^{1,1} \ar[dd]^{\Phi_6} \ar@/^/[rrd]^{d_2} & & & & &E_2^{3,0} \ar[dlll]^{\pi_*} \\
& & \overline{E}_2^{2,0} \ar[dd]^{\Phi_6} \ar[r]_{\beta_{3}}& \overline{E}_2^{3,0} \ar[dd]^{\Phi_6} & & & \\
\G2 \overline{E}(\Z_6)_2^{0,1}\ar@/^/[rrd]_{d_2} \ar[r]^{\beta_{3}} & \overline{E}(\Z_6)_2^{1,1} \ar@/^/[rrd]_(0.4){d_2}& & & & & \\
& & \overline{E}(\Z_6)_2^{2,0} \ar[r]_{\beta_{3}} & \overline{E}(\Z_6)_2^{3,0}& & &
} \\
\end{tabular} 
\end{center}
\caption{}\label{tab:d}
\end{table}
We have the 
equalities $\first_3 = w_4$ and $\second_3 \equiv v_2^3+z_6 \mod 3$. The generator of the module
$\Z_3 = H^3(\Z_6;M_n\otimes \Z_{(3)}) = E_2^{3,0}$ projects through $\pi_*$ to the Bockstein class $\beta_3 (w_4 (v_2^3+z_6)^{3^{i}k}) \in \overline{E}_2^{3,0}$
and hence the statement follows from comparing the differentials $d_2:\overline{E}_2^{0,1} \to \overline{E}_2^{2,0}$, that is
$$
d_2: H^0(\SL;M_n \otimes \Z_3) \to H^2(\SL;M_n \otimes \Z_3)
$$
and $d_2:  \overline{E}(\Z_6)_2^{0,1} \to  \overline{E}(\Z_6)_2^{2,0}$, that is
\begin{equation}
d_2: H^0(\Z_6;M_n \otimes \Z_3) \to H^2(\Z_6;M_n \otimes \Z_3).
\end{equation}
\end{proof}

\begin{thm} \label{t:ss_3tors}
We consider the extension problem for the $E^\infty$-term of the spectral sequence of lemma \ref{l:EinftyZ3} above. 
$K$ induced by the decomposition (\ref{e:decomp3}). 

For $n \equiv 0 \mod 12$ 
the summand $\Z_{3^{i+1}} \subset H^1(\SL;M_n\otimes \Z_{(3)})= E^2_{1,1}$ corresponding, via Universal Coefficient Theorem, 
to the generator $\second_3^{3^ih} \in H^0(\SL;M \otimes \Z_3)$ ($3 \nmid h$, $i>0$) 
has a lifting to the module $\Z_{3^{i+2}}$ through the extension with the module $\Z_3$ 
corresponding to the generator $$(p_2^3+q_6)^{3^{i}h} \in H^2(\Z_6;M \otimes \Z_{(3)}).$$

For all the other summands there is no lifting and the extension is trivial.
\end{thm}
\begin{proof}
We study the spectral sequence $\overline{E}_2$ converging to $H^*(B_3;M_n\otimes \Z_{3^k})$ for any $k>0$. As before, the spectral sequence collapses at $\overline{E}_3$.
We compute the cardinality of $ \overline{E}_3^{1,1} \oplus \overline{E}_3^{2,0}$ 
that is the cardinality of $H^2(B_3; M_n\otimes \Z_{3^k})$.
We start with an explicit description of 
$\overline{E}_2^{0,1}= H^0(\SL;M_n \otimes  \Z_{3^k})$,
$\overline{E}_2^{1,1}= H^1(\SL;M_n \otimes  \Z_{3^k})$, 
$\overline{E}_2^{2,0}= H^2(\SL;M_n \otimes  \Z_{3^k})$ and
$\overline{E}_2^{3,0}= H^3(\SL;M_n \otimes  \Z_{3^k})$ and of the corresponding $\overline{E}(\Z_6)_2$ terms of the 
spectral sequence for $H^2(\Z; M_n\otimes \Z_{3^k})$.

The term $\overline{E}(\Z_6)_2^{0,1} = H^0(\Z_6;M_n \otimes \Z_{3^k})$ is given by the ring of invariants $\Z_{3^k}[x,y]^{\Z_6}$. In particular we have the polynomial $q_6^{3^jh}$, for 
odd $h$, that 
generates a group $\Z_{3^k}$ and the polynomial $3^{k-1}w_4q_6^i$ that generates a group $\Z_3$; the groups are direct summands of $\Z_{3^k}[x,y]^{\Z_6}$.

The term $\overline{E}(\Z_6)_2^{1,1} = H^1(\Z_6;M_n \otimes \Z_{3^k})$ is a free $\Z_3$-module 
generated by the polynomials $3^{k-1}q_6^{3^jh}$ and by the image $\beta_{3^k}(3^{k-1}w_4q_6^i)$, where $\beta_{3^k}$ is the Bockstein map associated to the extension
$$\Z_{3^k} \to \Z_{3^{2k}} \to \Z_{3^k}.$$

The term $\overline{E}(\Z_6)_2^{2,0} = H^2(\Z_6;M_n \otimes \Z_{3^k})$ is a free $\Z_3$-module 
generated by the polynomials $3^{k-1}w_4q_6^i$ and $q_6^{3^jh} = \beta_{3^k}(3^{k-1}q_6^{3^jh})$.

The term $\overline{E}(\Z_6)_2^{3,0} = H^3(\Z_6;M_n \otimes \Z_{3^k})$ is a free $\Z_3$-module 
generated by the polynomials $3^{k-1}q_6^{3^jh}$ and $\beta_{3^k}(3^{k-1}w_4q_6^i)$.

The term $\overline{E}_2^{0,1} = H^0(\SL;M_n \otimes \Z_{3^k})$ is given by the ring of 
invariants $\Gamma_3[\mathcal{P}_3, \mathcal{Q}_3] \otimes \Z_{3^k}.$ 
In particular $3^{k-1}\mathcal{P}_3 \mathcal{Q}_3^i$ generates a group $\Z_3$ and 
$3^{(k-j-1) \vee 0}\mathcal{Q}_3^{3^jh}$ generates a group $\Z_{3^{k \wedge (j+1)}}$;
the groups are direct summands of $H^0(\SL;M_n \otimes \Z_{3^k})$. 

The term $\overline{E}_2^{1,1} = H^1(\SL;M_n \otimes \Z_{3^k})$ is given by the direct sum 
of a free $\Z_{3^k}$-module $\Z_{3^k}^{f_n}$
, the modules $\Z_3$ generated by classes $\overline{3^{k-1}q_6^i}$ such that $\Phi_6(\beta_{3^k}(\overline{3^{k-1}q_6^i})) = q_6^i \in H^2(\Z_6;M_n \otimes \Z_{3^k})$
and a module isomorphic to $\Gamma_3^+[\mathcal{P}_3, \mathcal{Q}_3] \otimes \Z_{3^k} 
\subset H^0(\SL;M_n \otimes \Z_{3^k})$.
%
%
%
%
In particular $\beta_{3^k}(3^{k-1}\mathcal{P}_3 \mathcal{Q}_3^i)$ generates a group $\Z_3$ and 
for $k > j$ the cocycle $\beta_{3^k}(3^{(k-j-1) \vee 0}\mathcal{Q}_3^{3^jh}) = 
\beta_{3^k}(3^{(k-j-1)}\mathcal{Q}_3^{3^jh})$ generates a group $\Z_{3^{k \wedge (j+1)}} = \Z_{3^{j+1}}$ and all these groups are direct summands in $\overline{E}_2^{1,1}$.

The inclusion $\Z_6 \into \SL$ induces for $i \geq 2$ the isomorphism $\Phi^i_6:H^i(\SL,M_n \otimes \Z_{3^k}) \simeq H^i(\Z_6;M_n \otimes \Z_{3^k})$. Hence the terms $\overline{E}_2^{2,0} = H^2(\SL; M_n \otimes \Z_{3^k})$ and $\overline{E}_2^{3,0} = H^3(\SL; M_n \otimes \Z_{3^k})$
are generated respectively by the counterimages of the polynomials $3^{k-1}w_4q_6^i$ and $q_6^{3^jh} = \beta_{3^k}(3^{k-1}q_6^{3^jh})$ and by the counterimages of  the polynomials $3^{k-1}q_6^{3^jh}$ and $\beta_{3^k}(3^{k-1}w_4q_6^i)$.

Now we compute the differentials $\overline{d}_2: \overline{E}(\Z_6)_2^{0,1} \to \overline{E}(\Z_6)_2^{2,0}$ and $\overline{d}_2: \overline{E}(\Z_6)_2^{1,1} \to \overline{E}(\Z_6)_2^{3,0}$
and the corresponding differentials for $\overline{E}_2$.

We use the fact that the spectral sequence $\overline{E}(\Z_6)$ collapses after the differential
$\overline{d}_2$ to the cohomology $H^*(\Z;M_n \otimes \Z_{3^k})$.  

For the computation of $\overline{d}_2: \overline{E}(\Z_6)_2^{0,1} \to \overline{E}(\Z_6)_2^{2,0}$
we have the maps $\overline{d}_2:q_6^{3^jh} \mapsto q_6^{3^jh}$ and $\overline{d}_2:3^{k-1}w_4q_6^i \mapsto 3^{k-1}w_4q_6^i$ that is the quotient map $\Z_{3^k} \to \Z_3$  and the isomorphism
$\Z_3 \to \Z_3$ for the corresponding generated submodules. From the Jordan blocks 
decomposition in the proof of proposition \ref{prop:invariants} one can see that for the other monomials in $\overline{E}(\Z_6)_2^{0,1}$ the differential is trivial.
%
%
%

The differential 
$\overline{d}_2: \overline{E}(\Z_6)_2^{1,1} \to \overline{E}(\Z_6)_2^{3,0}$ 
must be an isomorphism we have the maps 
$\overline{d}_2:3^{k-1}q_6^{3^jh} \mapsto 3^{k-1}q_6^{3^jh}$ and 
$\overline{d}_2:\beta_{3^k}(3^{k-1}w_4q_6^i) \mapsto \beta_{3^k}(3^{k-1}w_4q_6^i)$.

%
%

Finally we can describe the differential $\overline{d}_2$ for $\overline{E}$.
We use $\Phi_6: \overline{E}_2 \to \overline{E}(\Z_6)_2$ to compute it. 

Recall that we have the equalities $\mathcal{P}_3 = w_4$ and $\mathcal{Q}_3 \equiv q_6 + p_2^3 \mod 3$.

For the classes $3^{k-1}\mathcal{P}_3 \mathcal{Q}_3^i \in \overline{E}^{0,1}$ 
we have $\Phi_6 (3^{k-1}\mathcal{P}_3 \mathcal{Q}_3^i) = 3^{k-1}w_4q_6^i$ and $\overline{d}_2:3^{k-1}w_4q_6^i \mapsto 3^{k-1}w_4q_6^i$. Hence $\overline{d}_2 (3^{k-1}\mathcal{P}_3 \mathcal{Q}_3^i) = \Phi_6^{-1} (3^{k-1}w_4q_6^i)$.

For the classes $3^{(k-j-1) \vee 0}\mathcal{Q}_3^{3^jh} \in \overline{E}^{0,1}$ we have 
$\Phi_6 (3^{(k-j-1) \vee 0}\mathcal{Q}_3^{3^jh}) \equiv 
3^{(k-j-1) \vee 0}(q_6 + p_2^3)^{3^jh} \mod 3.$ Moreover 
$\overline{d}_2 ((q_6 + p_2^3)^{3^jh}) = q_6^{3^jh}$.
Hence $\overline{d}_2 (3^{(k-j-1) \vee 0}\mathcal{Q}_3^{3^jh}) =  3^{(k-j-1) \vee 0} \Phi_6^{-1}(q_6^{3^jh}) $ and since the cocycle $q_6^{3^jh}$ generates a $\Z_3$ module, 
the map is surjective if and only if $k \leq j+1$.
 
As the group $B_3$ had dimension $2$ the differential 
$\overline{d}_2: \overline{E}^{1,1} \to  \overline{E}^{3,0}$ must be surjective. 

As we did before for the $2$-torsion of higher order, the computation of the cardinality of $H^2(B_3; M_n\otimes \Z_{3^k})$ is
straightforward and we can define the function
$$ f(k):= \log_3(|H^2(B_3; M_n\otimes \Z_{3^k})|)
$$

As the spectral sequence $\overline{E}_r$ collapses at the term $\overline{E}_3$ we have
$
|H^2(B_3; M_n\otimes \Z_{3^k})| = |\overline{E}_3^{1,1} \oplus\overline{E}_3^{2,0}|
$ and hence:
\begin{equation} \label{e:log3}
f(k) = \epsilon_n(k) \log_3(|\Z_3|) + \log_3(|\Z_{3^k}^{f_n}|) + \log_3\left(\left| \oplus_{i,F} \Z_{3^{(i+1) \wedge k}}  \right|\right)
\end{equation}
The first summand in equation (\ref{e:log3}) corresponds to the submodule in $\overline{E}_2^{2,0}$ generated by 
$\Phi_6^{-1}(q_6^{3^jh})$ and since the generator survives to $\overline{E}_\infty$
if and only if $k > j+1$ we have that the coefficient $\epsilon_n(k)$ is given by
$$
\epsilon_n(k) = \left\lbrace 
\begin{array}{ll}
1 & \mbox{if }n = 12 \cdot 3^{j}h \mbox{ with }3 \nmid h \mbox{ and }  k > j+1, \\
0 & \mbox{otherwise.}
\end{array}
\right.
$$

The second summand in (\ref{e:log3}) corresponds to the submodule $\Z_{3^k}^{f_n} \subset \overline{E}_2^{1,1}.$

The third summand in (\ref{e:log3}) corresponds to the submodule of $\Gamma_3^+[\mathcal{P}_3, \mathcal{Q}_3] \otimes \Z_{3^k} \subset \overline{E}_2^{1,1}$
generated, for $i$ running over all positive integers, by the monic monomials in $2$ variables $F(x,y)$ with $\deg F(\mathcal{P}_3,\mathcal{Q}_3) = n$, such
that $F(\mathcal{P}_3,\mathcal{Q}_3)$ can be written as a monomial in $\mathcal{P}_3^{3^i}, \mathcal{Q}_3^{3^i}$, but not as a monomial
in $\mathcal{P}_3^{3^{i+1}}, \mathcal{Q}_3^{3^{i+1}}$  and $F(\mathcal{P}_3,\mathcal{Q}_3)$ is not in the form $\mathcal{P}_3\mathcal{Q}_3^j$. The direct sum runs over all these 
polynomials.

From (\ref{e:log3}) we get
$$
f(k) = \epsilon_n(k) + f_nk + {\sum}_{i,F} {(i+1) \wedge k}
$$
where the sum runs for $i,F$ as before.
It is clear that this is the logarithm of the cardinality of the degree $n$ part of a module
$N \otimes \Z_{3^k}$ where $N$ is given by
$$
\bigoplus_n \Z^{f_n} \oplus (\Ga^+_3[\first_3, \mathcal{Q}_3] \oplus \Z[\overline{\mathcal{Q}}_3] )/ \sim
$$
where $\overline{\mathcal{Q}}_3$ is a new variable of the same degree ($=12$)
as $\mathcal{Q}_3$ and the quotient module 
is defined by the relations $\frac{\mathcal{Q}_3^n}{n!} \sim 3 \overline{\mathcal{Q}}_3^n$ 
and $\first_3\frac{\mathcal{Q}_3^n}{n!} \sim 0$. 

From the description of the spectral sequence we can see the the extension involves the generator described in the statement.
\end{proof}
Hence we have proved also the following :
\begin{thm}\label{t:H2_3tors}
$$H^2(B_3; M)_{(3)} = (\Ga^+_3[\first_3, \mathcal{Q}_3] \oplus \Z[\overline{\mathcal{Q}}_3] )/ \sim$$ 
with $\deg \overline{\mathcal{Q}}_3 = 12$, 
where the quotient of modules is defined by 
$\frac{\mathcal{Q}_3^n}{n!} \sim 3 \overline{\mathcal{Q}}_3^n$ 
and $\first_3\frac{\mathcal{Q}_3^n}{n!} \sim 0$.
 \qed
\end{thm}
\end{subsection}

%% file: comparisons.tex
\section{Topological comparisons and speculations}\label{Topological comparisons and speculations}

The purpose of this section is to describe a topological space with the features that the reduced 
integral cohomology is precisely the torsion in $H^*(\SL;M \otimes \Z[1/6])$.

In particular, there are spaces $T_p(2n+ 1)$ which are total spaces of $p$-local fibrations, $p > 3$,
of the form $$S^{2n-1} \to T_p(2n+1) \to \Omega(S^{2n+1})$$
for which the integral cohomology is recalled next where it is assumed that $p>3$ throughout this section
(\cite{anick}, \cite{Gray-Theriault}).

The integral cohomology ring of  $\Omega(S^{2n+1})$ is the divided power algebra $\Gamma[x(2n)]$ where the
degree of $x(2n)$ is $2n$. The mod-$p$ homology ring (as a coalgebra)
of $T_p(2n+1)$ is $$E[v] \otimes F_p[w]$$ where the exterior algebra $E[v]$ is generated by an element $v$ of degree $2n-1$, and the mod-$p$ polynomial ring $ F_p[w]$ is generated by an element $w$ of degree $2n$.
Furthermore, the $r$-th higher order Bockstein defined 
by $b_{r}$ is given as follows:
\begin{equation*}
\begin{array}{l}
b_{r}(w^{p^s}) =\left\{\begin{array}{ll}
0 & \mbox{if } r < s \\
vw^{-1+ p^s} & \mbox{if } r= s.
\end{array}\right.\\
\end{array}
\end{equation*}

This information specifies the integral homology groups of $T_p(2n+1)$ 
(and thus the integral cohomology groups) listed in the following
theorem. Recall the algebra $\Gamma_p[x]$ of proposition \ref{prop:algebra}
In addition fix a positive $n$, and consider the graded ring $\Gamma_p[x(2n)]$
where $x(2n)$ has degree $2n > 0$.

\begin{thm} \label{thm:torsion.for.T}
 Assume that $p$ is prime with $p > 3$.
\begin{enumerate}

\item  The reduced integral homology groups of $T_p(2n+1)$ are given as follows:
\begin{equation*}
\begin{array}{l}
\Bar{H}_i(T_p(2n+1)) =\left\{\begin{array}{ll}
\Z/ p^r\Z & \mbox{if } i = 2np^{r-1}k-1 \mbox{ where }\ p \nmid k,\\
\{0\} & \mbox{otherwise.}
\end{array}\right.\\
\end{array}
\end{equation*}

\item The  additive structure for the cohomology ring of $T_p(2n+1)$ with $\Z_{(p)}$ coefficients 
is given as follows:
\begin{equation*}
\begin{array}{l}
\Bar{H}^i(T_p(2n+1)) =\left\{\begin{array}{ll}
\Z/ p^r\Z  & \mbox{if }i = 2np^{r-1}k \mbox{ where }p \nmid k,\\
\{0\}
& \mbox{otherwise.}
\end{array}\right.
\end{array}
\end{equation*}

\item The  cohomology ring of $T_p(2n+1)$ with $\Z_{(p)}$ coefficients 
is isomorphic to $\Gamma_p[x(2n)]$ in degrees strictly greater than $0$. \qed

\end{enumerate}

\end{thm}

The next step is to compare the cohomology of the spaces $T_p(2n+1)$
to the earlier computations.

\begin{thm} \label{thm:torsion.ala.shimura}
 Assume that $p$ is prime with $p > 3$. The $p$-torsion in the cohomology of 
$$E \times_{SL_2(\Z)} (\C\mathbb P^{\infty}) ^2$$ is the reduced cohomology of the suspension of the product of two spaces as follows:
$$\Sigma [   T_p(2p+3) \times T_p(2p^2 -2p + 1) ].$$

\vspace{-\baselineskip}
\qed

\end{thm}

%% file: appendix.tex
\appendix 
\section{Serre spectral sequence and Mayer-Vietoris exact sequence}
\label{appendix}
\begin{thm}[Proposition \ref{p:ssmayerviet}]
Let $G$ be any group of coefficients and let $$F \into E \stackrel{p}{\to} X$$ be a Serre fibration; suppose that the base space $X$ is the union of two open set $X = A \cup B$ and call $Y = A \cap B$ their intersection. 

Let $\mathrm{E}_2^{i,j}= H^i(X; H^j(F; G))$ (resp. ${\mathrm{E}(Y)}_2^{i,j}= H^i(Y; H^j(F; G))$) be the $\mathrm{E_2}$-term of the Serre spectral sequence for $F \into E \stackrel{p}{\to} X$ (resp. $F \into p^{-1}(Y) \stackrel{p}{\to} Y$).

The connecting homomorphism $$\delta^*: H^*(Y; H^j(F; G)) \to H^{*+1}(X; H^j(F; G))$$ in the Mayer-Vietoris long exact sequence associated to $X = A \cup B$ commutes with the spectral sequences differentials, hence it induces a map of spectral sequences $\delta^*: {\mathrm{E}(Y)} \to \mathrm{E}$. 
\end{thm}

Despite the previous theorem seems a quite natural result, we could not find a reference for it in the literature. We thank Pierre Vogel who suggested us how to prove it.

\begin{proof}
Define the space $X'$ homotopic equivalent to $X$ as
$$ X':= A \times \{ 0 \} \cup Y \times I \cup B \times \{ 1 \}$$
and 
$$ Z:= A \times \{ 0 \} \sqcup B \times \{ 1 \}.$$

We can consider the fibration $F \into E' \to X'$ induced by the homotopy 
equivalence $X \sim X'$ and the restriction $F \into E_Z \to A \sqcup B.$

We have the Serre spectral sequences
$$
H^i(X',Z;H^j(F;G)) \Rightarrow H^{i+j}(E',E_Z;G)
$$
and
$$
H^i(X';H^j(F;G)) \Rightarrow H^{i+j}(E';G).
$$

Consider the cohomology exact sequence for the couple $(E', E_Z)$:
\begin{equation*} 
\cdots \to H^{i}(E_Z;G) \to H^{i+1}(E',E_Z;G) \stackrel{\iota^*}{\to} H^{i+1}(E';G) \to 
\cdots.
\end{equation*}
The $\iota^*$ map in the corresponding sequence for the base spaces
\begin{equation} \label{eq:long2}
H^{i+1}(X',Z;H^j(F;G)) \stackrel{\iota^*}{\to} H^{i+1}(X';H^j(F;G)) 
\end{equation}
is a spectral sequence homomorphism as it is induced by the inclusion
$$
\xymatrix @R=1.2pc @C=0.9pc {
F \ar[r]^= \ar[d] & F \ar[d] \\
(E', \emptyset) \ar[r] \ar[d] & (E',E_Z)\ar[d] \\
(X', \emptyset) \ar[r] & (X', Z). \\
}
$$
Consider the right side of (\ref{eq:long2}). The homotopy equivalence $X \sim X'$ gives $$H^{i+1}(X';H^j(F;G)) \simeq H^{i+1}(X;H^j(F;G)).$$
On the left side of (\ref{eq:long2}) applying excision twice we have
$$
H^{i+1}(X',Z;H^j(F;G)) \simeq H^{i+1}(Y \times I,Y \times \{ 0 \} \sqcup Y \times \{ 1 \};H^j(F;G)) \simeq $$ $$ \simeq H^{i+1}(Y \times S^1,Y \times \{ 1 \};H^j(F;G)) \simeq H^i(Y;H^j(F;G)).
$$
where the last isomorphism follows by K\"unneth formula and is given by multiplication of the right side times the generator of the cohomology $H^1(S^1,\{ 1 \};G)$.

The proposition follows by the commutativity of the diagram
$$
\xymatrix @R=1.5pc @C=1pc {
H^{i+1}(X',Z;H^j(F;G)) \ar[r]^{\iota^*} \ar[d]^\simeq & H^{i+1}(X';H^j(F;G))\ar[d]^\simeq \\
H^{i}(Y;H^j(F;G)) \ar[r]^{\delta^*} & H^{i+1}(X;H^j(F;G)). \\
}
$$

\vspace{-\baselineskip}
\end{proof}